\RequirePackage{fix-cm}
\documentclass[smallcondensed]{svjour3}
\smartqed
\usepackage{graphicx}

\usepackage[T1]{fontenc}
\usepackage[utf8]{inputenc}
\usepackage{float}

\usepackage{amsmath}
\usepackage{amssymb}
\usepackage[english]{babel}

\makeatletter{}\newcommand{\dr}{\mathrm{d}}

\newcommand{\dd}[2]{\frac{\dr#1}{\dr#2}}
\newcommand{\pdd}[2]{\frac{\partial #1}{\partial #2}}

\newcommand{\ddn}[3]{\frac{\dr^{#1} #2}{\dr #3^{#1}}}
\newcommand{\pddn}[3]{\frac{\partial^{#1} #2}{\partial #3^{#1}}}

\newcommand{\bsprod}[2]{\left\langle #1,#2\right\rangle }
\newcommand{\evaluated}[2]{\left.#1\right|_{#2}}
\newcommand{\T}{\mathcal{T}}
\newcommand{\F}{\mathcal{F}}
\newcommand{\M}{\mathcal{M}}

\newcommand{\R}{\mathbb{R}}

\newcommand{\finalTime}{ t_f }

\newtheorem{thm}{\protect\theoremname}
\newtheorem{lem}[thm]{\protect\lemmaname}

\addto\captionsenglish{\renewcommand{\lemmaname}{Lemma}}
\addto\captionsenglish{\renewcommand{\propositionname}{Proposition}}
\addto\captionsenglish{\renewcommand{\theoremname}{Theorem}}
\providecommand{\lemmaname}{Lemma}
\providecommand{\propositionname}{Proposition}
\providecommand{\theoremname}{Theorem}
 
\begin{document}


\title{Higher Order Cut Finite Elements for the Wave Equation
\thanks{This research was supported by the Swedish Research Council Grant No. 2014-6088.}}
\author{Simon Sticko \and Gunilla Kreiss}

\institute{
S. Sticko \at
\email{simon.sticko@it.uu.se}
Division of Scientific Computing, Department of Information
Technology, Uppsala University, Sweden
ORCiD: 0000-0002-4694-4731
\and
G. Kreiss \at
\email{gunilla.kreiss@it.uu.se}
Division of Scientific Computing, Department of Information
Technology, Uppsala University, Sweden
}

\maketitle

\begin{abstract}
The scalar wave equation is solved using higher order immersed finite elements.
We demonstrate that higher order convergence can be obtained.
Small cuts with the background mesh are stabilized by adding penalty terms to the weak formulation.
This ensures that the condition numbers of the mass and stiffness matrix are independent of how the boundary cuts the mesh.
The penalties consist of jumps in higher order derivatives integrated over the interior faces of the elements cut by the boundary.
The dependence on the polynomial degree of the condition number of the stabilized mass matrix is estimated. 
We conclude that the condition number grows extremely fast when increasing the polynomial degree of the finite element space.
The time step restriction of the resulting system is investigated numerically and is concluded not to be worse than for a standard (non-immersed) finite element method.
\end{abstract}

\makeatletter{}\section{Introduction}
Cut elements \cite{CutFEM2014} is an immersed finite element method. For a
domain immersed in a background mesh, one solves for the degrees of freedom of the
smallest set of elements covering the domain. The inner products in the weak form are
taken over the immersed domain. That is, on each element one integrates over the part
of the element that is inside the domain. As a result of this, some elements will
have a very small intersection with the immersed domain. This will make some
eigenvalues of the discrete system very small and in turn, result in poorly
conditioned matrices. A suggested way to remedy this is by adding stabilizing terms to the weak formulation.
A jump-stabilization was suggested in \cite{Burman2012} for the case of piecewise
linear elements, where the jump in the normal derivative is integrated over the faces of
the elements intersected by the boundary. 
This stabilization makes it possible to prove that the condition numbers of the involved matrices are bounded independently of how the boundary cuts the elements.
This form of stabilization has been used with
good results in several recent papers, see for example
\cite{CutFEM2014,Hansbo2014,MassingStokes,sticko2016}, and has also been used for PDEs posed on surfaces in \cite{hansbo_characteristic_2015,Burman2016}.

Thus, a lot of attention has been directed to the use of lower order elements.
Higher order cut elements have received less attention so far.
These are interesting in wave propagation problems.
The reason for this is that the amount of work per dispersion error typically increases slower for higher order methods for this type of problems.
In \cite{MassingStokes} it was suggested to stabilize higher order elements by integrating also
jumps in higher derivatives over the faces. This is the intuitive higher
order generalization of the stabilization first suggested in \cite{Burman2012} and
was also mentioned as a possibility in \cite{BurmanGhost}.

In this paper, we consider solving the scalar wave equation using higher order cut elements.
Both the mass and stiffness matrix are stabilized using the higher order jump-stabilization.
We present numerical results showing that the method results in higher order convergence rate.
The time-step restriction of the resulting system is computed numerically and is concluded to be of the same size as for standard finite elements with aligned boundaries.
Furthermore, we estimate how the condition number of the stabilized mass matrix depend on the polynomial degree of the basis functions.
The estimate suggests that the condition number grows extremely fast with respect to the polynomial degree, which is supported by the numerical experiments.
As a remedy for this behavior, we consider lowering the order of the elements close to the boundary.
This results in a better condition number, but also in at least half an order lower convergence compared to having elements of full order everywhere.
All numerical experiments are performed in two dimensions, but the generalization to three dimensions is immediate.

One reason why the considered stabilization is attractive is because it is quite easy to implement.
Integrals over internal faces occur also in discontinuous Galerkin methods, thus making the implementation similar to what is already supported in many existing libraries.

The suggested jump-stabilization is one but not the only possibility for stabilizing an immersed method.
In \cite{Johansson2013} a higher order discontinuous Galerkin method was suggested and proved to give optimal order of convergence.
Here the problem of ill-conditioning was solved by associating elements that had small intersections with neighboring elements.
Similar approaches has been used with higher order elements in for example \cite{kummer_extended_2017,muller_high-order_2016}, where elements with small intersection take their basis functions from an element inside the domain.
One problem with these approaches is that it is not obvious how to choose which elements should merge with or associate to one another.
A related alternative to these is the approach in \cite{reusken_analysis_2008}, where individual basis functions were removed if they have a small support inside the domain.
A different approach was used in \cite{Johansson2015} where streamline diffusion stabilization was added to the elements intersected by the boundary.
This was proved to give up to fourth order convergence.
However, this approach is restricted to interface problems.
Another alternative is to use preconditioners to try to overcome problems with ill-conditioning, such as in \cite{lehrenfeld_optimal_2016}.
However, only preconditioning does not solve the problem of severe time-step restrictions when using explicit time-stepping.
For this reason preconditioning alone is not sufficient in the context of wave-propagation.

This paper is organized in the following way. Notation and some basic problem setup
is explained in Section \ref{sec:notation}, the stabilized weak formulation is described
in Section \ref{sec:jump_stabilization}, and the stability of the method is discussed in Section \ref{sec:stability}.
Analysis of how fast the condition number increases when increasing the polynomial degree is presented in Section
\ref{sec:estimates}, and numerical experiments is presented in Section \ref{sec:numerics}.

\section{Problem Statement and Theoretical Considerations\label{sec:theory}}
\makeatletter{}\subsection{Notation and Setting \label{sec:notation}}
Consider the wave equation
\begin{align}
\ddot{u}&= \nabla^2 u+f(x,t) \, &x\in \Omega,   \quad  &t\in[0,\finalTime], \label{eq:waveEquation}\\
u&= g_D(x,t)                  \, &x\in \Gamma_D, \quad  &t\in[0,\finalTime], \nonumber\\
\pdd{u}{n}&= g_N(x,t)         \, &x\in \Gamma_N, \quad  &t\in[0,\finalTime],\nonumber\\
u&=u_0(x)                    \, &x\in \Omega,   \quad  &t=0, \nonumber\\
\dot{u}&=v_0(x)              \, &x\in \Omega,   \quad  &t=0, \nonumber
\end{align}
posed on a given domain $\Omega$, with $\Gamma_D\cup \Gamma_N=\partial \Omega$.
Let $\Omega\subset \R^d$ be immersed in a mesh, $\T$, as in
Figure \ref{fig:immersed_domain}. We assume that each element $T\in\T$ has some part
which is inside $\Omega$, that is: $T\cap \Omega\neq \emptyset$. Furthermore, let
$\Omega_\T$ be the domain that corresponds to $\T$, that is 
\begin{equation*}
\Omega_\T=\bigcup_{T\in \T} T.
\end{equation*}
Let $\T_\Gamma$ denote the set
of elements that are intersected by $\partial \Omega$:
\begin{equation*}
\T_\Gamma=\{T\in\T : T \cap \partial\Omega\neq \emptyset\},
\end{equation*}
as in Figure \ref{fig:cut_mesh}. 
Let $\F_\Gamma$ denote the faces seen in Figure \ref{fig:stablized_faces}.
That is, the faces of the elements in $\T_\Gamma$, excluding the faces that make up $\partial \Omega_\T$.
To be precise, $\F_\Gamma$ is defined as
\begin{equation*}
\F_\Gamma=\{F=T_1 \cap T_2 \; : \; T_1 \in \T_\Gamma \;\text{ or }\; T_2 \in \T_\Gamma, \quad T_1,T_2\in \T  \}. 
\end{equation*}

We assume that our background mesh is sufficiently fine, so that the immersed geometry is well resolved by the mesh.
Furthermore, we shall restrict ourselves to meshes as the one in Figure
\ref{fig:immersed_domain}, where we have a mesh consisting of hypercubes and our
coordinate axes are aligned with the mesh faces. That is, the face normals have a
nonzero component only in one of the coordinate directions. 
Denote the distance between two grid points in any coordinate direction by $h$.

Consider the situation in Figure \ref{fig:Elements}, where two neighboring elements, $T_1$
and $T_2$, are sharing a common face $F$.
Denote by $\partial_n^k v$ the $k$th directional derivative in the direction of
the face normal.
That is, fix $j\in\{1,\ldots,d\}$ and let the normal of the face, $n$, be such that
\begin{equation}
n_i=
\begin{cases}
\pm 1 & i=j \\
0 & i\neq j,
\end{cases}
\label{eq:kronecker_normal}
\end{equation}
then define
\begin{equation}
\partial_{n}^{k}v = n_j^k \pddn{k}{v}{x_j}.
\label{eq:normal_derivative}
\end{equation}

In the following we denote by $(\cdot,\cdot)_X$ and $\bsprod{\cdot}{\cdot}_Y$ the
$L_2$ scalar products taken over the $d$ and $d-1$ dimensional domains $X\subset \R^d$
and $Y\subset \R^{d-1}$. Let $\| \cdot \|_Z$ denote the corresponding norm,
and let $|\cdot|_{H^s(Z)}$ denote the $H^s$-semi-norm.
By $[v]$ we shall denote a jump over a face, $F$, that is: $[v]=v|_{F_+}-v|_{F_-}$.

We shall assume that our basis functions are
tensor products of one-dimensional polynomials of order $p$. In particular,
we shall use Lagrange elements with Gauss-Lobatto nodes, in the
following referred to as $Q_p$-elements, $p\in \{1,2,\ldots\}$.
Let $V_h^p$ denote a continuous finite element space, consisting of $Q_p$-elements on the mesh $\T$:
\begin{equation}
V_h^p=\left \{ v\in C^0(\Omega_\T): \evaluated{v}{T}\in Q_p(T) \right \}.
\label{eq:fullFEMspace}
\end{equation}
Define also the following semi-norm
\begin{equation*}
| v |_{\star}^2 = \| \nabla v \|_{\Omega_\T}^2 + \frac{1}{h} \| v \|_{\Gamma_D}^2,
\end{equation*}
which is a norm on $V_h^p$ in the case that $\Gamma_D\neq \emptyset$.

\def \figuresize {.15\paperwidth}
\begin{figure}[H] 
\centering
\begin{minipage}[t]{\figuresize} 
\includegraphics[width=\columnwidth]{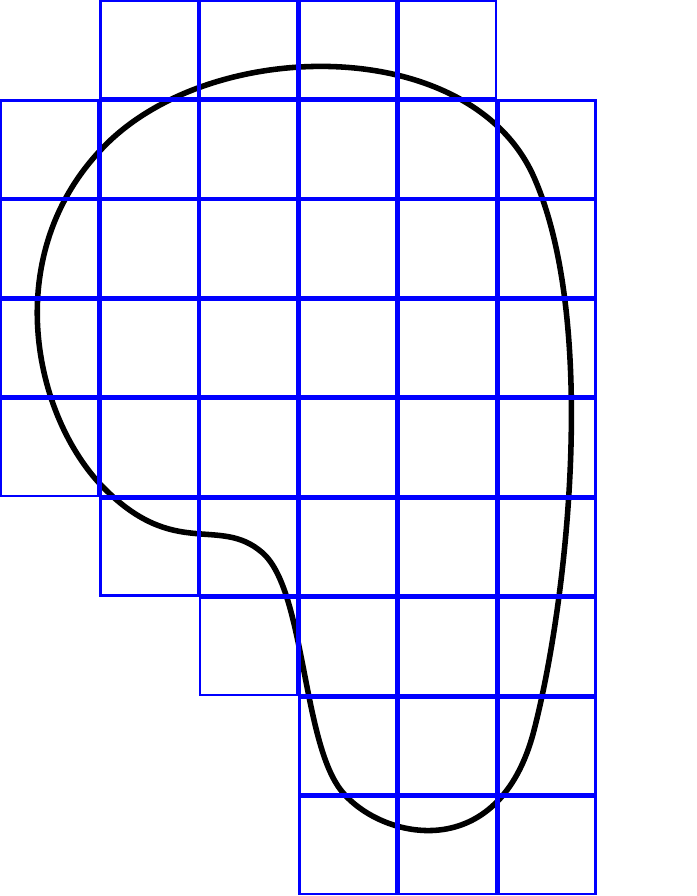}
\caption{$\Omega$ immersed in a mesh $\T$, covering $\Omega_\T$}
\label{fig:immersed_domain}
\end{minipage}
\hfill
\begin{minipage}[t]{\figuresize}
\includegraphics[width=\columnwidth]{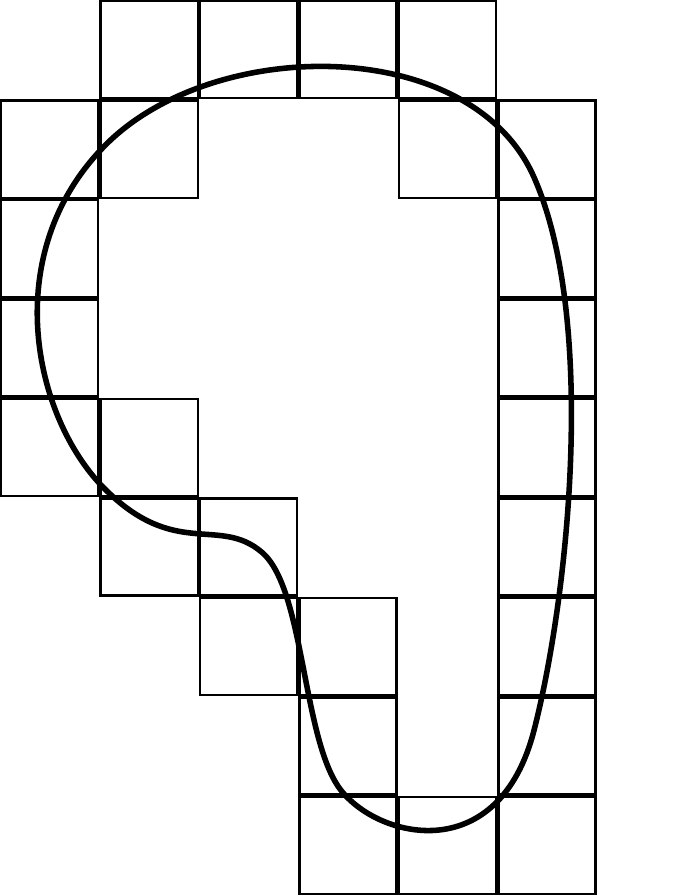}
\caption{Intersected elements $\T_\Gamma$}
\label{fig:cut_mesh}
\end{minipage}
\hfill
\begin{minipage}[t]{\figuresize}
\includegraphics[width=\columnwidth]{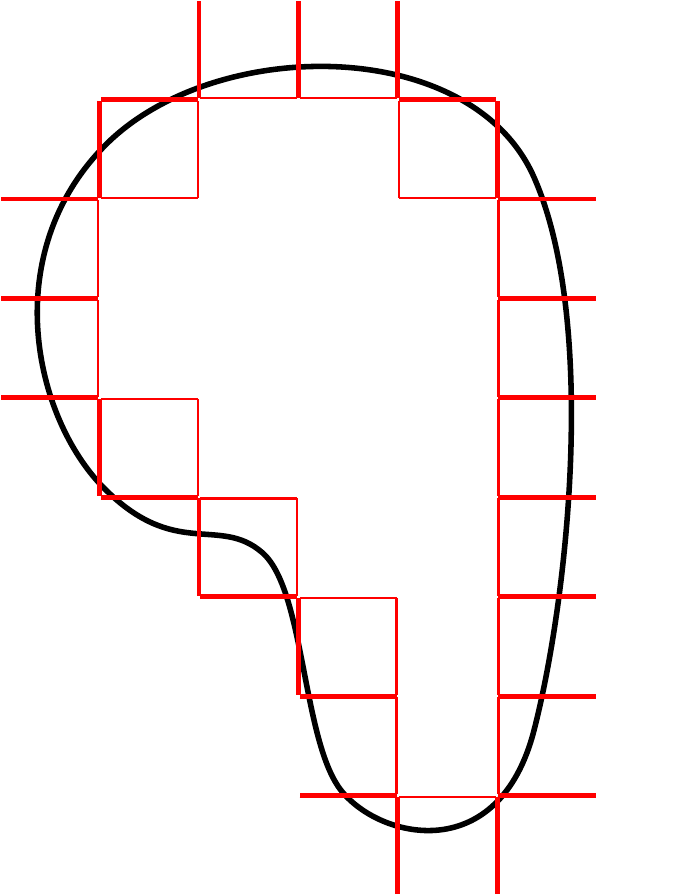}
\caption{Faces $\F_\Gamma$}
\label{fig:stablized_faces}
\end{minipage}
\end{figure}

\begin{figure}[H] \centering
\includegraphics[width=.2\paperwidth]{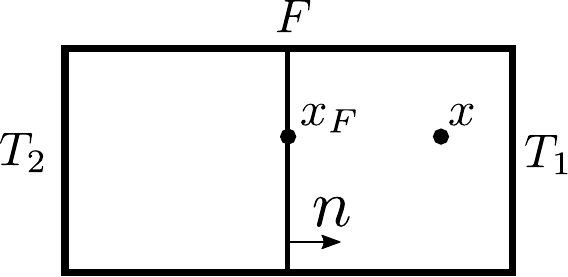}
\caption{Two elements sharing a common face}
\label{fig:Elements}
\end{figure} 
\makeatletter{}\subsection{The Stabilized Weak Formulation \label{sec:jump_stabilization}}
Multiplying \eqref{eq:waveEquation} by a test-function, integrating by parts,
and applying boundary conditions by Nitsche's method \cite{Nitsche1971} leads to a 
weak formulation of the following form:
find $u_h$ such that for each fix $t\in (0,\finalTime]$, $u_h\in V_h^p$ and
\begin{equation}
(\ddot{u}_h,v)_\Omega+a(u_h,v)=L(v), \quad \forall v\in V_h^p,
\label{eq:naive_weak_form}
\end{equation}
where
\begin{equation*}
a(u_h,v)=(\nabla u_h, \nabla v)_\Omega -\bsprod{\pdd{u_h}{n}}{v}_{\Gamma_D}
-\bsprod{u_h}{\pdd{v}{n}}_{\Gamma_D}+\frac{\gamma_{D}}{h} \bsprod{u_h}{v}_{\Gamma_D},
\end{equation*}
\begin{equation*}
L(v)=(f,v)_\Omega
+\bsprod{g_{D}}{\frac{\gamma_{D}}{h} v-\pdd{v}{n}}_{\Gamma_D}
+\bsprod{g_{N}}{v}_{\Gamma_{N}}.
\end{equation*}
What makes this different from standard finite elements is that the integration on each element needs to be adapted to the part of the element that is inside the domain. 
As illustrated in Figure \ref{fig:small_cut}, some elements will have a very small intersection with the domain.
Consider the mass-matrix in \eqref{eq:naive_weak_form}:
\begin{equation}
\tilde{\mathcal{M}}_{ij}=(\phi_i,\phi_j)_\Omega.
\label{eq:unstabilized_mass_matrix}
\end{equation}
Note that its smallest eigenvalue is smaller than each diagonal entry:
\begin{equation}
\lambda_{\min} =\min_{
z \in \R^N \, : \, z \neq 0
}  \frac{z^T \tilde{\mathcal{M}} z}{z^T z}\leq \tilde{\mathcal{M}}_{ii}, \quad i=1,\ldots,N.
\end{equation}
Depending on the size of the cut with the background mesh some diagonal entries can become arbitrarily close to zero.
Thus, both the mass and stiffness matrix can now be arbitrarily ill-conditioned depending on how the cut occurs.
Because of this, one can not guarantee that the method is stable.
\begin{figure}[H]
\centering
\includegraphics[width=.12\paperwidth]{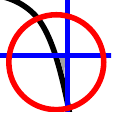}
\includegraphics[width=.12\paperwidth]{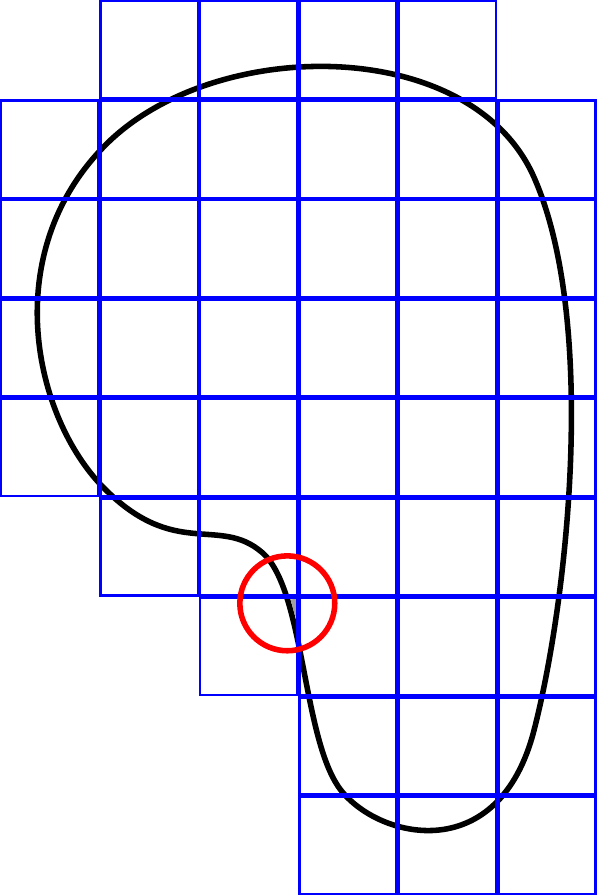} 
\caption{An element having a small intersection (in gray) with the domain}
\label{fig:small_cut}
\end{figure}
One way to try to remedy this is by adding stabilizing terms, $j$, to the two bilinear forms
\begin{equation}
\begin{split}
M(\ddot{u}_h,v)=& (\ddot{u}_h,v)_\Omega+\gamma_M j(\ddot{u}_h,v),\\
A(u_h,v)=& a(u_h,v)+\gamma_A h^{-2} j(u_h,v),
\end{split}
\label{eq:stabilizedForms}
\end{equation}
where $\gamma_M,\gamma_A>0$ are penalty parameters.
This gives us the following weak formulation: 
find $u_h$ such that for each fix $t\in (0,\finalTime]$, $u_h\in V_h^p$ and
\begin{equation}
M(\ddot{u}_h,v)+A(u_h,v)=L(v), \quad \forall v\in V_h^p.
\label{eq:weak_form}
\end{equation}
In \cite{MassingStokes} a stabilization term of the following form was suggested
\begin{equation}
j(u,v)=\sum_{F\in \F_\Gamma} 
\sum_{k=1}^p h^{2k+1}\bsprod{[\partial_n^k
u]}{[\partial_n^k v]}_F,
\label{eq:basic_stabilization}
\end{equation}
which in some sense is the intuitive extension of the stabilization which was
suggested in \cite{Burman2012}.
The stabilization in \eqref{eq:basic_stabilization} was also briefly mentioned as
a possibility in \cite{BurmanGhost}. As was discussed in \cite{MassingStokes} the bilinear form \eqref{eq:stabilizedForms} can be
shown to define a scalar product which is norm equivalent to the $L_2$-norm on the
whole background mesh:
\begin{equation}
C_L\| v \|_{\Omega_\T}^2\leq M(v,v)\leq C_U \| v \|_{\Omega_\T}^2, \quad \forall v\in V_h^p,
\label{eq:norm_equivalence}
\end{equation}
and a corresponding equivalence also holds for the gradient:
\begin{equation}
\tilde{C}_L\| \nabla v \|_{\Omega_\T}^2\leq \| \nabla v  \|_\Omega^2 + \gamma_A h^{-2} j(v,v) \leq \tilde{C}_U \| \nabla v \|_{\Omega_\T}^2, \quad \forall v\in V_h^p.
\label{eq:stiffness_equivalence}
\end{equation}

The constants in \eqref{eq:norm_equivalence} and \eqref{eq:stiffness_equivalence} depend on the
polynomial degree of our basis functions, but not on how the boundary cuts
through the mesh.
Let $\M$ denote the mass matrix with respect to the bilinear form $M$, and $\M_\T$
with respect to the scalar product on the background mesh, that is:
\begin{equation*}
\begin{split}
\M_{ij}&=M(\phi_i,\phi_j), \\
(\M_\T)_{ij}&=(\phi_i,\phi_j)_{\Omega_\T}.
\end{split}
\end{equation*}
Now, \eqref{eq:norm_equivalence} implies that the condition number, $\kappa(\M)$, of $\M$ is bounded by the condition number of $\M_\T$:
\begin{equation}
\kappa(\M)\leq \frac{C_U}{C_L} \kappa (\M_\T).
\label{eq:condInTermsOfLowerAndUpper}
\end{equation}
The property \eqref{eq:stiffness_equivalence} is necessary in order to show that $A(\cdot,\cdot)$ is coercive in $V_h^p$ with respect to the $| \cdot |_\star$-semi-norm on the background mesh:
\begin{equation}
\exists C_c > 0 \, \text{:} \quad C_c | v |_\star^2 \leq A(v,v), \quad \forall v\in V_h^p.
\label{eq:coercive}
\end{equation}
As we shall see in Section \ref{sec:stability} this is needed in order to show that the method is stable with respect to time.
The result in \eqref{eq:coercive} follows by the same procedure as in \cite{MassingStokes}, assuming that the following inverse inequality holds
\footnote{This was also shown for piecewise linear basis functions in the proof of Lemma 4 in \cite{hansbo_hansbo_2002}.}
\begin{equation}
h^{1/2} \left \| \pdd{v}{n} \right \|_{\Gamma \cap T} \leq C p \| \nabla v \|_{T}, \quad \forall v\in V_h^p.
\label{eq:inverseImmersedNormalDeriv}
\end{equation}
This inequality follows the same scaling with respect to $h$ and $p$ as the corresponding standard inverse inequality, which relates the norm over a face to the norm over the whole element.
See for example \cite{warburtonConstantsHP}. 

The stabilization in \eqref{eq:basic_stabilization} is the basic form of stabilization that we shall consider. 
However, each time we differentiate we will introduce some dependence on the polynomial degree. 
It therefore seems reasonable that each term in the sum should be scaled in some way. 
Because of this, we consider a stabilization of the following form:
\begin{equation}
j(u,v)=\sum_{F\in \F_\Gamma} 
\sum_{k=1}^p w_k\frac{h^{2k+1}}{(2k+1)(k!)^2}\bsprod{[\partial_n^k u]}{[\partial_n^k v]}_F,
\label{eq:jump_stabilization}
\end{equation}
where $w_j\in \R^+$ are some weights, which we are free to choose as we wish.
The choice of weights will determine how large our constants $C_U$, $C_L$ in
\eqref{eq:condInTermsOfLowerAndUpper} are, and in turn influence how well conditioned the mass matrix is.
 
\makeatletter{}\subsection{Stability \label{sec:stability}}
The bilinear forms in \eqref{eq:weak_form} are symmetric.
This is a quite important property, since this in the end will guarantee stability of the system.
In order to show stability we want a bound over time on $\| u \|_{\Omega_\T}$.
Define an energy, $E$, of the form
\begin{equation}
E(t):=\frac{1}{2} \left ( M(\dot{u}_h,\dot{u}_h) + A(u_h,u_h)  \right).
\end{equation}
Since both bilinear forms are at least positive semi-definite, this energy has the property $E \geq 0$.
The symmetry now allows us to show that for a homogeneous system,
\begin{equation*}
\begin{split}
f(x)=0, \quad
g_D(x)=0, \quad
g_N(x)=0,
\end{split}		
\end{equation*}
the energy is conserved:
\begin{equation}
\dd{E}{t}=M(\ddot{u}_h,\dot{u}_h)+A(u_h,\dot{u}_h)\overset{\eqref{eq:weak_form}}{=}0,
\label{eq:basicEnergy}
\end{equation}
so that
\begin{equation}
E(t)=E(0).
\label{eq:EnergyConservation}
\end{equation}
By the definition of the energy together with \eqref{eq:norm_equivalence} and \eqref{eq:coercive} this immediately implies that $\| \dot{u}_h\|_{\Omega_\T}$ and $\| \nabla u_h \|_{\Omega_\T}$ are both bounded.
For the case $\Gamma_D\neq \emptyset $ the semi-norm $| \cdot |_\star$ is a norm for the space $V_h^p$ and \eqref{eq:coercive} implies that
$\| u_h \|_{{\Omega_\T}} $  is also bounded.
When $\Gamma_D = \emptyset$ we can use that
\begin{equation*}
2 \| u_h \|_{\Omega_\T} \dd{}{t} \| u_h \|_{\Omega_\T}= \dd{}{t} \|u_h\|_{\Omega_\T}^2 =2(u_h,\dot{u}_h)_{\Omega_\T} \leq 2 \|u_h\|_{\Omega_\T}  \|\dot{u}_h\|_{\Omega_\T},
\end{equation*}
which gives us
\begin{equation*}
\dd{}{t} \| u_h \|_{\Omega_\T} \leq \| \dot{u}_h \|_{{\Omega_\T}}.
\end{equation*}
By integrating we obtain that $\| u_h \|_{{\Omega_\T}} $  is bounded since $\| \dot{u}_h \|_{{\Omega_\T}} $ is bounded:
\begin{equation*}
\| u_h (t) \|_{\Omega_\T} \leq \| u_h(0) \|_{\Omega_\T} + \int_0^{\finalTime}\| \dot{u}_h \|_{{\Omega_\T}} dt.
\end{equation*}
Thus the system is stable.

In total the system \eqref{eq:weak_form} discretizes to a system of the form
\begin{equation}
\mathcal{M}\ddn{2}{\xi}{t}+\mathcal{A}\xi=\mathcal{F}(t),
\label{eq:discreteSystem}
\end{equation}
with $\mathcal{M},\mathcal{A}\in \R^{N \times N}$, and $\xi\in \R^N$, and where 
\begin{equation*}
\mathcal{A}_{ij}=A(\phi_i,\phi_j).
\end{equation*}
When solving this system in time we will have a restriction on the time-step, $\tau$, of the form
\begin{equation}
\tau \leq \alpha  C_{FL} h,
\label{eq:time_step_bound}
\end{equation}
where $\alpha$ is a constant which depends on the time-stepping algorithm.
If we for example use a classical 4th-order explicit Runge-Kutta $\alpha=2\sqrt{2}$.
The constant $C_{FL}$ is given by
\begin{equation}
C_{FL}=\frac{h^{-1}}{\sqrt{\lambda_{\max}}},
\label{eq:cflNumber}
\end{equation}
where $\lambda_{\max}$ is the largest eigenvalue of the generalized eigenvalue problem:
find $(x,\lambda)$ such that
\begin{equation*}
\mathcal{A} x -\lambda \mathcal{M}x=0, \quad x\in \R^N.
\end{equation*}
One would expect that the added stabilization has some effect on the $C_{FL}$-constant.
Because of this, we will investigate this constant experimentally in Section \ref{sec:numerics}.
It turns out that the $C_{FL}$-constant is not worse than for the standard case (with boundaries aligned with the mesh). 
\makeatletter{}\subsection{Analysis of the Condition Number of the Mass Matrix\label{sec:estimates}}
We would like to choose the weights in \eqref{eq:jump_stabilization} in order to minimize the condition number of the mass matrix.
This is particularly important when it comes to wave-propagation problems.
For this application one typically uses an explicit time-stepping method.
When this is the case we need to solve a system involving the mass matrix in each time-step.

In order to choose the weights we need to know how the condition number depends on the weights and the polynomial degree.
To determine this, we follow essentially the same path as in \cite{MassingStokes} and keep track of the weights and the polynomial dependence of the involved inequalities.
In the following, we denote by $C$ various constants which do not depend on $h$ or $p$, unless explicitly stated otherwise.
We shall also by $w$ denote the vector $w=(w_1,\ldots,w_p)$, where $w_j$ are the weights in the stabilization term \eqref{eq:jump_stabilization}.
We can now derive the following inequality, which is a weighted version of Lemma 5.1
in \cite{MassingStokes}.
\begin{lem}
\label{lem:fundamental_inequality}
Given two neighboring elements, $T_1$ and $T_2$, sharing a face $F$ (as in Figure
\ref{fig:Elements}), and $v\in V_h^p$, we have that:
\begin{equation}
\| v \|_{T_1}^2 \leq 
{L(w)} \left ( \| v \|_{T_2}^2 +  \sum_{k=1}^p w_k
\frac{h^{2k+1}}{(2k+1)(k!)^2} \| [\partial_n^k v] \|_{F}^2 \right ),
\label{eq:fundamental_inequality}
\end{equation}
where
\begin{equation}
{L(w)}=C_1(p) + \sum_{k=1}^p \frac{1}{w_k}.
\label{eq:definition_L}
\end{equation}
\end{lem}
\begin{proof}
Denote by $v_i$ the restriction of $v$ to $T_i$ and then extended by expression to the whole of $T_1 \cup T_2$.
As in Figure \ref{fig:Elements}, let $x\in T_1$ and denote by $x_F(x)$ the projection of $x$ onto the face.
Let $n$ be the normal pointing towards $T_1$ and let $j$ denote the only nonzero component, as in \eqref{eq:kronecker_normal}.
We may now Taylor expand from the face:
\[
v_i(x)=
\sum_{k=0}^p \frac{1}{k!} \partial_j^k v_i(x_F(x)) (x_j-x_{F,j})^k.
\]
Using that
\begin{equation*}
x_j-x_{F,j}=n_j|x-x_F|
\end{equation*}
gives us

\[
v_i(x)=
\sum_{k=0}^p \frac{1}{k!} \partial_n^k v_i(x_F(x)) |x-x_F(x)|^k,
\]
by definition of  $\partial_n^k v$ from \eqref{eq:normal_derivative}.
Consequently we have that
\begin{equation}
v_1(x)=v_2(x)+\sum_{k=1}^p \frac{1}{k!} [\partial_n^k v(x_F)] |x-x_F|^k.
\label{eq:taylordifference}
\end{equation}
Now introduce the following weighted $l^2(\R^{p+1})$-norm:
\begin{equation*}
\| z \|_\alpha^2 :=\sum_{k=0}^p \alpha_k z_k^2,
\end{equation*}
where $\alpha_k>0$ and $z\in \R^{p+1}$. If $\|\cdot \|_1$ denotes the usual
$l^1(\R^{p+1})$-norm we have that:
\begin{equation}
\| z  \|_1^2\leq C_\alpha \| z\|_\alpha^2,
\label{eq:Wnorm_equiv}
\end{equation}
where 
\begin{equation}
C_\alpha=\sum_{k=0}^p \frac{1}{\alpha_k}.
\end{equation}
Taking the $L_2(T_1)$-norm of \eqref{eq:taylordifference} and using
\eqref{eq:Wnorm_equiv} now results in:
\begin{equation}
\| v_1 \|_{T_1}^2 \leq 
C_\alpha \left ( \alpha_{0} \| v_2 \|_{T_1}^2 +  \sum_{k=1}^p \alpha_k
\frac{h^{2k+1}}{(2k+1)(k!)^2} \| [\partial_n^k v] \|_{F}^2 \right ).
\label{eq:after_W_norm}
\end{equation}
Since $v_2$ lies in a finite dimensional polynomial space on $T_1 \cup T_2$ the
norms on $T_1$ and $T_2$ are equivalent:
\[
\| v_2 \|_{T_1}^2 \leq C_1 \| v_2 \|_{T_2}^2,
\]
where $C_1=C_1(p)$.
Using this in \eqref{eq:after_W_norm} and choosing
\begin{equation*}
\begin{split}
 \alpha_{0}&=1/C_1, \\
 \alpha_{k}&=w_k, \quad k=1,\ldots,p
\end{split}		
\end{equation*}
gives us \eqref{eq:fundamental_inequality}.
\end{proof}
Lemma \ref{lem:fundamental_inequality} will now allow us to give a lower bound on the bilinear form $M$, which was defined in \eqref{eq:stabilizedForms}.
\begin{lem}
\label{lem:lower_bound_Mform}
A lower bound for $M(v,v)$ is:
\begin{equation}
\| v \|_{\Omega_\T}^2 \leq C_l {L(w)}^{N_J} M(v,v),
\label{eq:lower_bound_Mform}
\end{equation}
where $L(w)$ is given by \eqref{eq:definition_L}, $N_J$ is some sufficiently large integer and  $C_l$ is a constant independent of $h$ and $p$.
\end{lem}
\begin{proof}
Let $T_0\in \T_\Gamma$ and let $\{T_i\}_{i=1}^{N-1}$ (with $T_i\in \T_\Gamma $)
be a sequence of elements that need to be crossed in order to get to an element 
$T_N\in \T\setminus \T_\Gamma$
as in Figure \ref{fig:ElementWalk}, and denote $F_i=T_{i-1}\cap T_i$.
By using \eqref{eq:fundamental_inequality} we get
\begin{equation*}
\|v \|_{T_0}^2 \leq {L(w)}^N \left ( \|v \|_{T_N}^2+ \sum_{i=1}^N  
\sum_{k=1}^p w_k
\frac{h^{2k+1}}{(2k+1)(k!)^2} \| [\partial_n^k v] \|_{F_i}^2
 \right),
\end{equation*}
where we have used that ${L(w)}\geq 1$ (since at least $C_1\geq 1)$. 
Let now $N_J\geq 1$ denote some upper bound on the maximum number of jumps that needs to be made in the mesh. 
If our geometry is well resolved by our background mesh $N_J$ is a small integer. 
This gives us
\[
\|v\|_{\Omega_\T}^2=
\sum_{T\in \T_\Gamma} \|v\|_{T}^2+\sum_{T\in \T\setminus \T_\Gamma} \|v\|_{T}^2
\leq C {L(w)}^{N_J} \left ( \sum_{T\in \T\setminus \T_\Gamma} \|v\|_{T}^2 + j(v,v)
\right ),
\]
from which \eqref{eq:lower_bound_Mform} follows.
\end{proof}
\begin{figure}[H]
\centering
\includegraphics[width=.15\paperwidth]{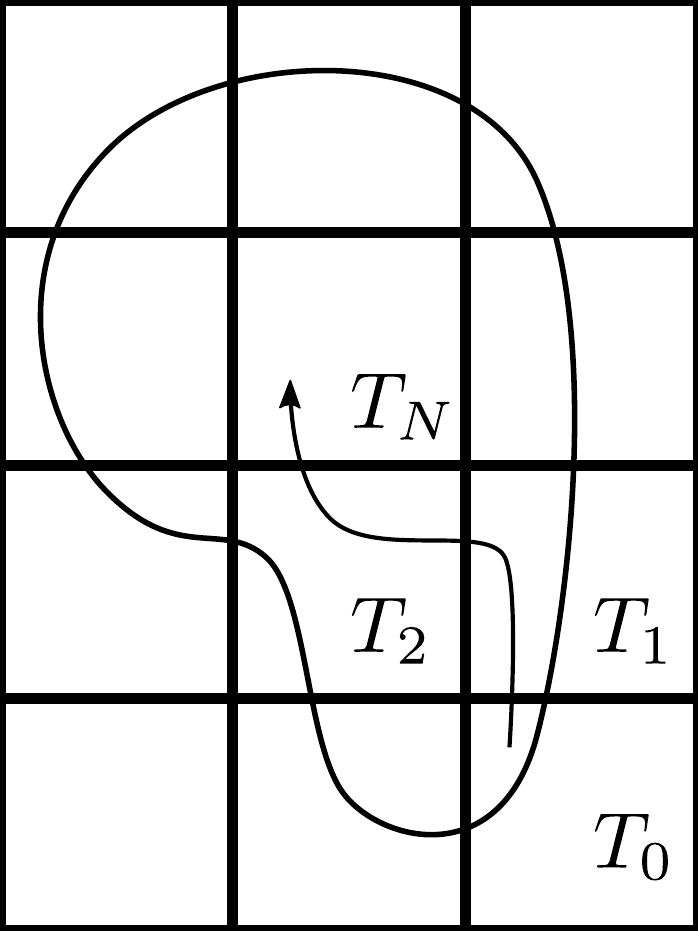}
\caption{A sequence of jumps from a boundary element $T_0\in \T_\Gamma$ to an inside element
$T_N$}
\label{fig:ElementWalk}
\end{figure}

We proceed by estimating how a bound on the jumps depends on the polynomial degree.
\begin{lem}
\label{lem:facejump}
For the jumps in the normal derivative we have that:
\begin{equation}
\| [\partial_n^k v]  \|_F^2 \leq C_k \frac{p^{4k+2}}{h^{2k+1}} \left( \| v
\|_{T_F^+}^2+\| v \|_{T_F^-}^2 \right ), \quad \text{for} \quad k=1,2,\ldots,p
\label{eq:facejump} 
\end{equation}
where $T_F^+$ and $T_F^-$ denotes the two elements sharing the face $F$.
\end{lem}
\begin{proof}
Note first that
\begin{equation}
\| [\partial_n^k v]  \|_F^2 \leq 2 
\left ( \| \partial_n^k v_1  \|_{F}^2+\| \partial_n^k v_2 \|_{F}^2 \right).
\label{eq:facejump_ineq} 
\end{equation}
We shall need the following inequalities:
\begin{equation}
\| v \|_F \leq C\frac{p}{\sqrt{h}} \| v \|_T,
\label{eq:face_to_element}
\end{equation}
\begin{equation}
| v |_{H^s(T)} \leq C^s\frac{p^{2s}}{h^{s}} \| v \|_T,
\label{eq:higher_to_lower}
\end{equation}
which were discussed\footnote{In particular see (4.6.4) and (4.6.5) in Theorem 4.76, together with the reasoning leading to Corollary 3.94} in \cite{Schwab1998}.
Although \eqref{eq:higher_to_lower} holds for a whole element we shall use the
corresponding inequality applied to a face:
\begin{equation}
| v |_{H^s(F)} \leq C^s\frac{p^{2s}}{h^{s}} \| v \|_F.
\label{eq:higher_to_lower_face}
\end{equation}
This is valid since a function $v$ in the tensor product space over $T$ will have a restriction $\evaluated{v}{F}$
in the tensor product space over the face $F$. 
Note that the constants, $C$, in \eqref{eq:higher_to_lower} and \eqref{eq:higher_to_lower_face} are not necessarily the same.
By combining
\eqref{eq:facejump_ineq}, \eqref{eq:face_to_element} and \eqref{eq:higher_to_lower_face} we obtain
\eqref{eq:facejump}.
\end{proof}
Using Lemma \ref{lem:facejump} we can now bound the bilinear form $M(\cdot,\cdot)$ from
above.
\begin{lem}
\label{lem:upper_bound_Mform}
An upper bound for $M(v,v)$ is:
\begin{equation}
M(v,v)\leq ( 1+C_g G(w) )\| v \|_{\Omega_\T}^2,
\label{eq:upper_bound_Mform}
\end{equation}
where
\begin{equation}
G(w)=\sum_{k=1}^p w_k\frac{p^{4k+2}}{(2k+1)(k!)^2},
\end{equation}
and $C_g$ is a constant independent of $h$ and $p$.
\end{lem}
\begin{proof}
Using the definition of $j(\cdot,\cdot)$ and applying Lemma \ref{lem:facejump} on
each order of derivatives in the sum individually we have
\begin{equation*}
j(v,v)\leq C G(w)\sum_{F\in \mathcal{F}_\Gamma} 
\left ( \| v \|_{T_F^-}^2 + \| v \|_{T_F^+}^2 \right).
\end{equation*}
Let $n_F$ denote the number of faces that an element has in $\R^d$. We now
have 
\begin{equation}
\sum_{F\in \F_\Gamma} \left ( \| v \|_{T_F^+}^2 + \| v \|_{T_F^-}^2 \right ) 
\leq 
2n_F \sum_{T\in \T} \| v \|_{T_F}^2
\leq 2n_F \| v \|_{\Omega_\T}^2,
\label{eq:faces2mesh}
\end{equation}
so we finally obtain:
\[
j(v,v)\leq C_g G(w)\| v \|_{\Omega_\T}^2,
\]
which gives us \eqref{eq:upper_bound_Mform}.
\end{proof}
By using Lemma \ref{lem:lower_bound_Mform} and \ref{lem:upper_bound_Mform} we now have the following bound on the 
condition number.
\begin{lem}
\label{lem:condition_bound}
An upper bound for the condition number of the mass matrix is
\begin{equation}
\kappa(\mathcal{M})\leq C_M(w) \kappa(\mathcal{M^*}),
\label{eq:condition_bound}
\end{equation}
where
\[
C_M=C_l {L(w)}^{N_J} (1+C_g G(w)) \kappa(\M^*).
\]
\end{lem}
\begin{proof}
Let $\lambda(\cdot)$ denote eigenvalues.
From Lemma \ref{lem:lower_bound_Mform} and \ref{lem:upper_bound_Mform} we obtain 
\[
\frac{\lambda_{\min} (\M^*)}{{C_l L(w)}^{N_J}} \leq \lambda_{\min}(\M),
\]
\[
\lambda_{\max}(\M) \leq (1+C_g G(w)) \lambda_{\max} (\M^*),
\]
which gives us \eqref{eq:condition_bound}.
\end{proof}

Here, we would like to choose the weights in order to minimize the constant $C_M$.
However, we have the following unsatisfactory result, which shows that no matter how we choose the weights our analysis cannot yield a $p$-independent bound on the conditioning. 
\begin{lem}
The constant $C_M(w)$ in Lemma \ref{lem:condition_bound} fulfills $C_M(w)\geq C_0 P(p)$, where $C_0$ does not depend on $p$ or $w$.
Here $P(p)$ is the function 
\begin{equation}
P(p)=\sum_{k=1}^p \frac{p^{4k+2}}{(k!)^2(2k+1)},
\label{eq:Pfunction}
\end{equation}
which is independent of the choice of weights $w$.
\label{lem:conditioning}
\end{lem}
\begin{proof}
First note that
\begin{equation*}
C_l {L(w)}^{N_J} (1+C_g G(w))\geq C_l C_g G(w) {L(w)}^{N_J}\geq C_l C_g L(w) G(w) .
\end{equation*}
Now we have
\begin{multline*}
{L(w)} G(w)\geq
 \sum_{k=1}^p \left(w_k \frac{p^{4k+2}}{(2k+1)(k!)^2} \right)
\sum_{k=1}^p \left(\frac{1}{w_k} \right)\geq\\
\sqrt{\sum_{k=1}^p \left(w_k^2 \left(\frac{p^{4k+2}}{(2k+1)(k!)^2}\right)^2 \right)}
\sqrt{\sum_{k=1}^p \left(\frac{1}{w_k} \right)^2}\geq P(p),
\end{multline*}
where we first used that the $l^1(\R^{p})$-norm is greater than the $l^2(\R^{p})$-norm and finally Cauchy-Schwartz. From this the result follows.
\end{proof}

The function $P(p)$ increases incredibly fast when increasing the polynomial degree. This result could reflect either:
\begin{enumerate}
  \item The analysis leading to Lemma \ref{lem:condition_bound} is not sharp. The bound $C_M$ is too generous, and a better bound exists.
  \item The bound in Lemma \ref{lem:condition_bound} is not unnecessarily generous, so that the constant $C_M$ is in some sense ``tight''.
  This means that the condition number of the stabilized mass matrix \eqref{eq:stabilizedForms} will grow faster than the function $P(p)$, regardless of the choice of weights.
\end{enumerate}
Alternative 2 is rather devastating from a time-stepping perspective, since in order to time-step \eqref{eq:discreteSystem} an inverse of the mass matrix needs to be available in each time-step.
If this inversion is done with an iterative method the number of required iterations until convergence is going to be large.

A combination of these two alternatives is, of course, possible.
The estimate in Lemma \ref{lem:condition_bound} could be too pessimistic, but even the optimal bound increases incredibly fast.
Given the results in Section \ref{sec:numerics} this appears to be the most plausible alternative.

\makeatletter{}\subsection{Lowering the Order at the Boundary\label{sec:lowerAtBoundary}}
To remedy the expected poor behavior of the condition number of the mass matrix we shall consider lowering the order of the elements close to the boundary.
This will be done in the way illustrated in Figure \ref{fig:lowerOrderAtBoundary}.
This idea is based on two observations:
\begin{itemize}
  \item In finite difference methods it is possible to lower the order close to
  the boundary and still get full convergence \cite{gustafsson_convergence_1975,siyang_convergence_2015}.
  \item By using lower order elements close to the boundary we only need to stabilize jumps in derivatives up to order $p-1$.
\end{itemize}
Let $N_F(T)$ denote the neighboring element of the element $T$ sharing the face $F$ with $T$.
We now construct a new finite element space $\tilde{V}_h^p$ in the following way. Elements which are intersected or have an intersected neighbor are lowered one order compared to the interior of the domain.
More precisely:
\begin{equation}
\tilde{V}_h^p=\left \{v\in C^0(\Omega_\T):
\begin{cases}
\evaluated{v}{T}\in Q_{p-1}(T),\quad &T\in \T_\Gamma \text{ or } \exists F : N_F(T)\in \T_\Gamma\\
\evaluated{v}{T}\in Q_p(T),\quad &\text{Otherwise}\\
\end{cases}		
\right \}.
\label{eq:reducedFEMspace}
\end{equation}
Using this space we can still guarantee stability. 

The space in \eqref{eq:reducedFEMspace} will introduce hanging nodes between elements of different orders.
This can be solved in several ways, but in the experiments in Section \ref{sec:numerics} we treat this by adding constraints that enforce continuity at the hanging nodes.

\begin{figure}[H]
\centering
  \includegraphics[width=.3\paperwidth]{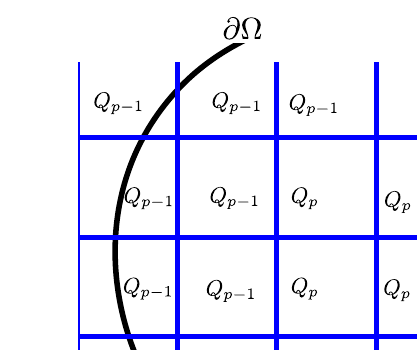}
  \caption{Illustration of the space $\tilde{V}_h^p$, that was constructed by lowering the order of the
  elements close to the boundary.}
  \label{fig:lowerOrderAtBoundary}
\end{figure}

\makeatletter{}\subsection{Choosing Weights in the Jump-Stabilization}
In order to do a computation, we are forced to make some choice of the weights $w_i$.
The essence of Lemma \ref{lem:conditioning} is that we can bound $L(w)G(w)$ from below. 
So in order to choose weights let us assume that:
\begin{equation*}
\kappa(\M)\propto L(w)G(w).
\end{equation*}
From Lemma \ref{lem:upper_bound_Mform} it is seen that choosing $w_i\gg 1$ makes
$G(w)$ very large. In the same way, Lemma \ref{lem:lower_bound_Mform} tells us
that choosing $w_i \ll 1$ for some $i$ makes $L(w)$ very large.
From this observation it seems reasonable to try to enforce both bounds to be
of about the same magnitude. In this way, we minimize $L(w)G(w)$ with respect to
$w$ and enforce $G(w)=L(w)$. This leaves us with
\begin{equation*}
\nabla_w L+\nabla_w G=0, 
\end{equation*}
where $\nabla_w$ denotes the gradient with respect to $w$.
This now gives us the following choice of weights:
\begin{equation}
w_k=k!\frac{\sqrt{2k+1}}{p^{2k+1}}.
\label{eq:mini_weights}
\end{equation}
There is no reason why this argumentation should lead to the optimal choice of
weights, but it seems reasonable that this is \emph{not a particularly bad choice}.

\section{Numerical Experiments \label{sec:numerics}}
\makeatletter{}In the following, we shall solve both an inner problem and an outer problem using finite element spaces of different orders. 
For the inner problem we use both the standard spaces $V_h^p$, defined by \eqref{eq:fullFEMspace}, and $\tilde{V}_h^p$, the corresponding spaces with lower polynomial order in the elements close to the boundary.
For the outer problem we only use the space $V_h^p$. 

The weights from \eqref{eq:mini_weights} are used, with $p$ determined by the order of the polynomials at the boundary.
In addition, the following parameters are used
\begin{equation*}
\begin{split}
\gamma_M=&0.25\sqrt{3},\\
\gamma_A=&0.5\sqrt{3},\\
\gamma_D=&5p^2.
\end{split}		
\end{equation*}
The scaling of $\gamma_D$ with respect to $p$ follows from the inequality \eqref{eq:inverseImmersedNormalDeriv}.
When ${p=1}$ these parameters coincide with the parameters used in \cite{sticko2016}. There the effect of $\gamma_M$ on the condition number of the mass matrix was investigated numerically.
For $p=1$ this choice of $\gamma_A$ and $\gamma_D$ also coincides with the one in \cite{Burman2012}, where $\gamma_A$ was investigated numerically.

The errors are computed in norms which are some quantities integrated over the domain $\Omega$.
It is worth noting that the geometry of $\Omega$ is represented by a level set function, $\psi_h$.
Both for the case when $u\in V_h^p$ and when $u\in \tilde{V}_h^p$ the level set function is an element in the space 
\begin{equation*}
W_h^p=\{v \in C_0(\Omega_{\T_B}) : \evaluated{v}{T}=Q_p(T) \},
\end{equation*}
where 
\begin{equation*}
\Omega_{\T_B}=\bigcup_{T\in \T_B} T,
\end{equation*}
and where $\T_B$ is a larger background mesh from which $\T$ was created.
In order to perform the quadratures over the elements intersected by the boundary we have used the algorithm in \cite{Saye2015},
which generates the quadrature rules on the intersected elements with respect to $\psi_h$.
Thus, also the errors of the solution are calculated with respect to this approximation of the geometry. That is, the $L_2$-norms are approximated as
\begin{equation*}
\begin{split}
\| \cdot \|_\Omega \approx \| \cdot \|_{\psi_h<0}, \\
\| \cdot \|_{\partial \Omega} \approx \| \cdot \|_{\psi_h=0},
\end{split}
\end{equation*}
where $\psi_h$ is initialized by $L_2$-projecting the analytic level set function onto the space $W_h^p$.
Convergence-rates are estimated as
\begin{equation*}
\frac{\log(e_{i} / e_{i+1})}{\log(h_{i} / h_{i+1})},
\end{equation*}
where $e_i$ denotes an error corresponding to mesh size $h_i$.

Time-stepping is performed with a classical fourth order explicit Runge-Kutta, after rewriting the system \eqref{eq:discreteSystem} as a first order system in time.
A time step, $\tau$, of size
\begin{equation}
\tau = \frac{0.4}{p^2} h
\label{eq:used_time_step}
\end{equation}
is used.
During the time-stepping we need to solve a system involving the mass matrix.
When using higher order elements the condition number of the mass matrix is large, so an iterative method requires a lot of iterations.
Because of this, a direct solver was used.
On non-cut elements Gauss-Lobatto quadrature was used to assemble the mass matrix.
This makes the mass matrix almost diagonal, which makes use of a direct solver cheaper.
All off-diagonal entries in the mass matrix are related to degrees of freedom close to the immersed boundary.

The library \texttt{deal.II} \cite{dealII84} was used to implement the method.

\subsection{Standard Reference Problem with Aligned Boundary \label{sec:alignedBoundaries}}
It is relevant to compare some of the properties of the mass and stiffness matrix with standard (non-immersed) finite elements.
The unstabilized mass and stiffness matrix were computed on a rectangular grid with size $[-1.5,1.5] \times [-1.5,1.5]$, with Neumann boundary conditions.
As for the immersed case, quadrilateral Lagrange elements with Gauss-Lobatto nodes were used.
The computed $C_{FL}$ is shown in Table \ref{tab:cfl}.
The $C_{FL}$-number was computed according to \eqref{eq:cflNumber}.
For a given $p$ the value in Table \ref{tab:cfl} is the mean value when calculating the CFL-number over a number of grid sizes.
The condition number of the mass matrix is shown in Figure \ref{fig:condnumbers} and the minimal and maximal eigenvalues of the mass matrix is shown in Figure \ref{fig:EigenvaluesMassMatrix}.
Since all eigenvalues should be proportional to $h^2$, the eigenvalues have been scaled by $h^{-2}$ for easier comparison.

\begin{table}[H]
\caption{Computed CFL-numbers for the non-immersed case in Section \ref{sec:alignedBoundaries} and 
the immersed inner problem in Section \ref{sec:innerproblem}}
\center
\makeatletter{}\begin{tabular}{|c|c|c|}
\hline
\textbf{p}&\textbf{Non-Immersed}&\textbf{Immersed}\\\hline
1&0.20&0.34\\\hline
2&0.09&0.10\\\hline
3&0.05&0.05\\\hline
\end{tabular}
 
\label{tab:cfl}
\end{table}

\begin{figure}[H]
\centering
\includegraphics[width=.45\paperwidth]{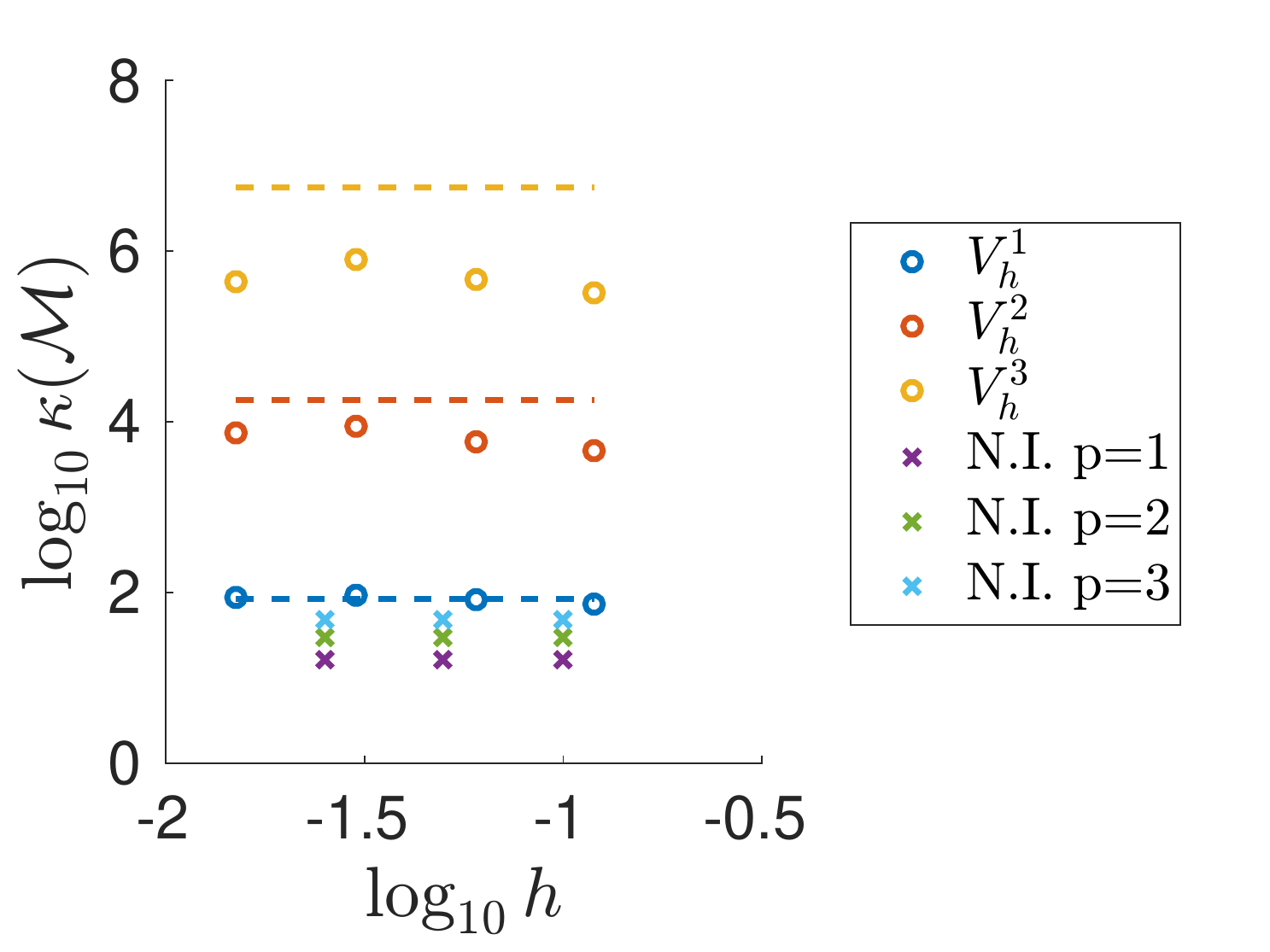}
\caption{Condition number of the mass matrix for the non-immersed (N.I.) problem in Section \ref{sec:alignedBoundaries}
and the immersed inner problem in Section \ref{sec:innerproblem}.
For different $h$ and $p$. The dashed lines denotes estimates according to the function $P(p)$.}
\label{fig:condnumbers}
\end{figure}

\begin{figure}[H]
\newcommand{\massEigFigWidth}{.7\paperwidth}
\centerline{
\begin{minipage}{\massEigFigWidth}
\includegraphics[width=.45\columnwidth]{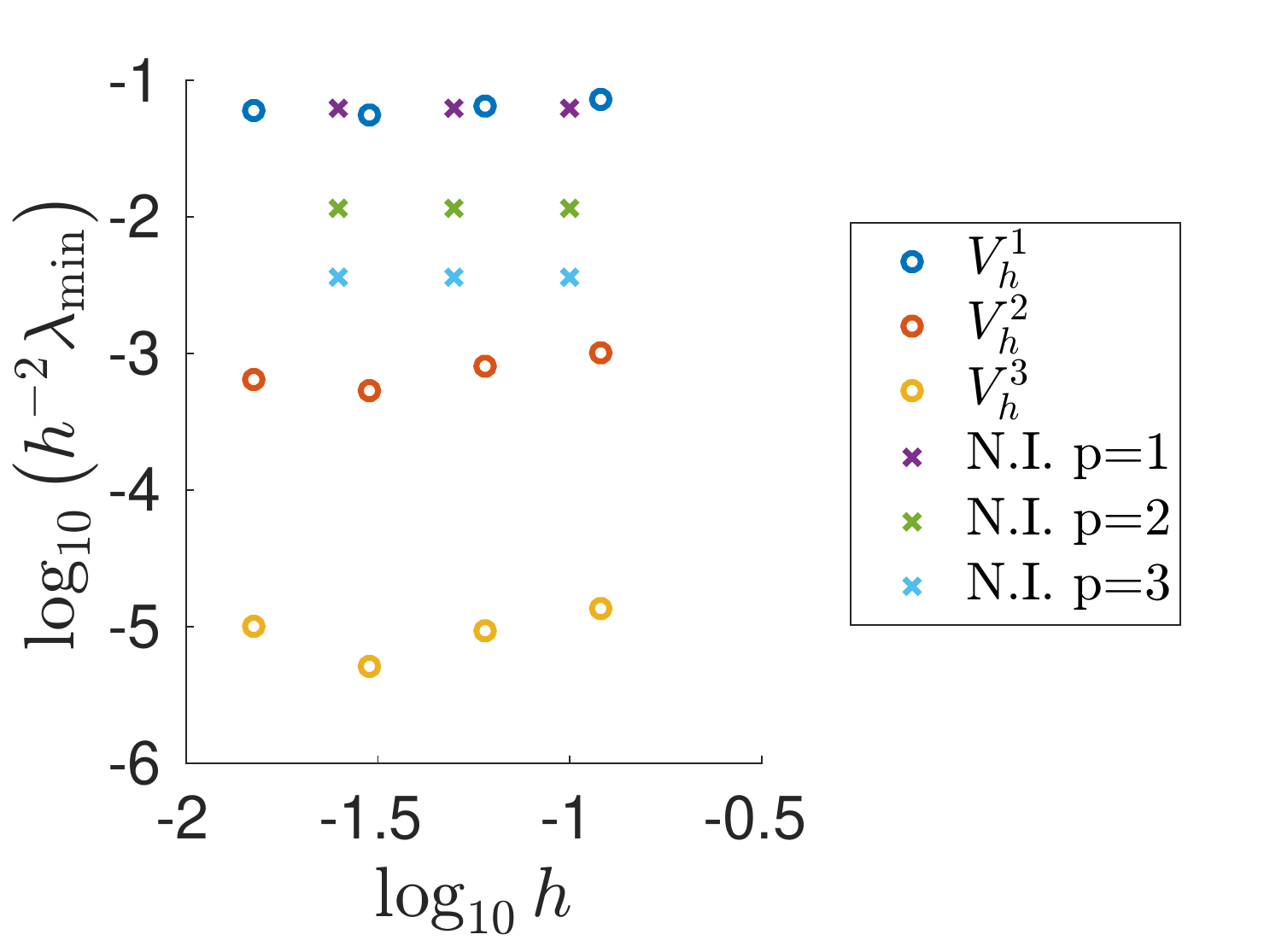}
\includegraphics[width=.45\columnwidth]{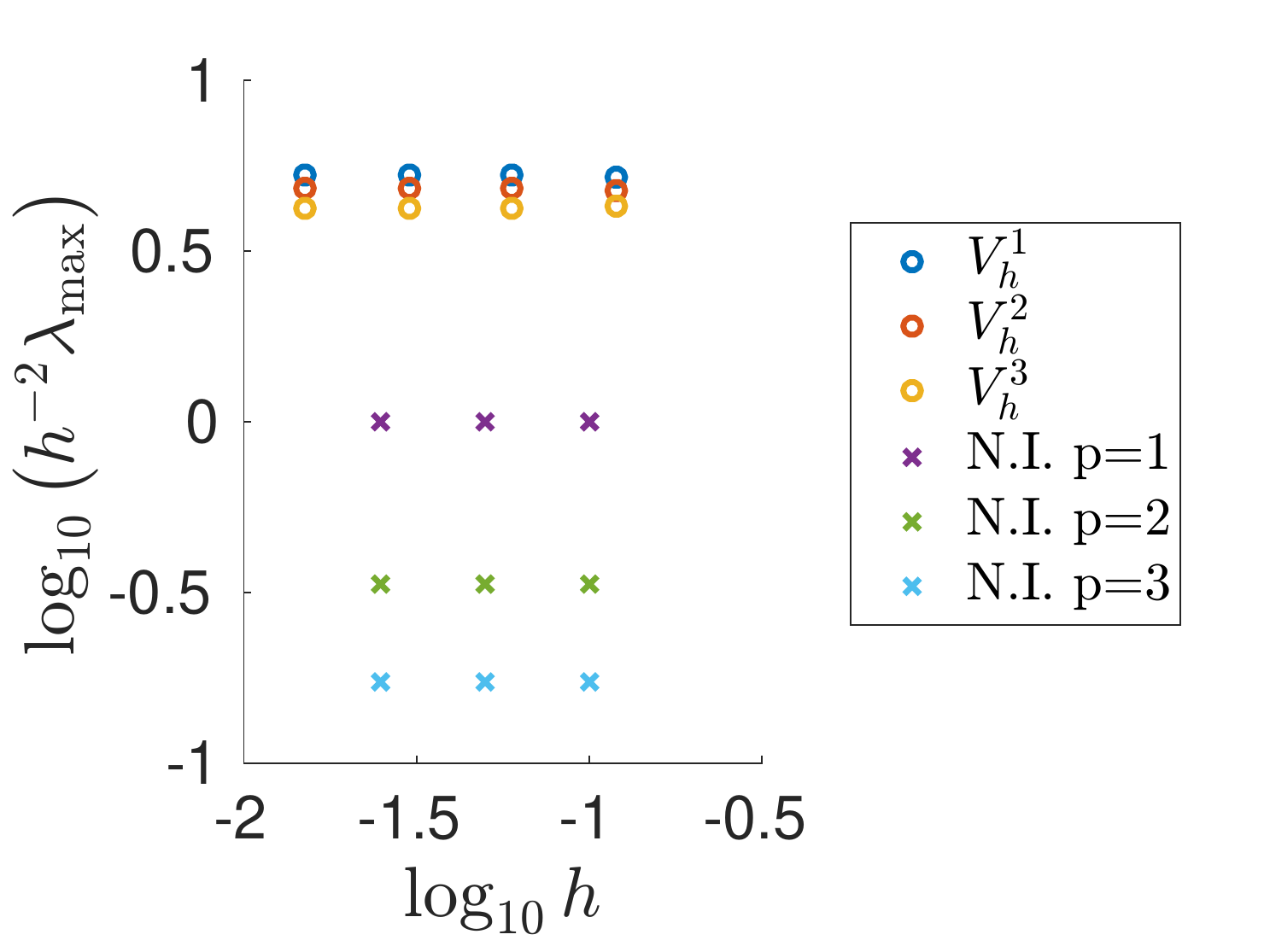}
\end{minipage}
}
\caption{Minimal/maximal eigenvalues (scaled by $h^{-2}$) of the mass matrix, for the non-immersed (N.I.) problem in Section \ref{sec:alignedBoundaries} and the immersed inner problem in Section \ref{sec:innerproblem}. 
For different $h$ and $p$. \label{fig:EigenvaluesMassMatrix}}
\end{figure} 
\makeatletter{}\subsection{An Immersed Inner Problem \label{sec:innerproblem}}
Let $\Omega$ be a disk domain, centered at origo, with radius $R=1$, and enforce homogeneous Dirichlet boundary condition along the boundary
\begin{equation*}
\evaluated{u}{\partial \Omega}=0.
\end{equation*}
Let $J_0$ denote the 0:th order Bessel-function and let $\alpha_n$ denote its
$n$:th zero. By starting from initial conditions:
\begin{equation*}
\begin{split}
\evaluated{u}{t=0}&=J_0(\alpha_n \frac{\| x \|}{R}),\\
\evaluated{\pdd{u}{t}}{t=0}&=0,
\end{split}		
\end{equation*}
we can calculate the error in our numerical solution with respect to the
analytic solution:
\begin{equation*}
u(x,t)= J_{0} ( \alpha_{n} \frac{\|x \|}{R} ) \cos(\omega_{n}t), 
\quad \omega_{n}=\frac{\alpha_{n}}{R}.
\label{eq:AnalyticSol}
\end{equation*}
A few snapshots of the numerical solution are shown in Figure \ref{fig:innerSnapshots}.
The problem was solved with the given method until an end-time, $\finalTime$, corresponding to a three periods:
\begin{equation*}
\finalTime=3 T_p, \quad T_p=\frac{2 \pi}{\omega_n}.
\end{equation*}
At the end-time the errors were computed.
\begin{figure}[H]
\newcommand{\innerSnapWidth}{.25\paperwidth}
\centering
  \begin{minipage}{\innerSnapWidth}
	\includegraphics[width=\innerSnapWidth]{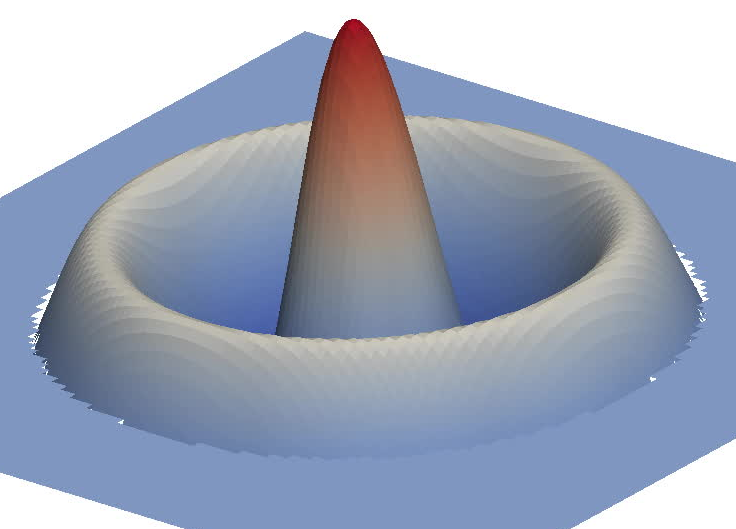}
	\centerline{$t=0$}
  \end{minipage}
  \quad 
  \begin{minipage}{\innerSnapWidth}
	\includegraphics[width=\innerSnapWidth]{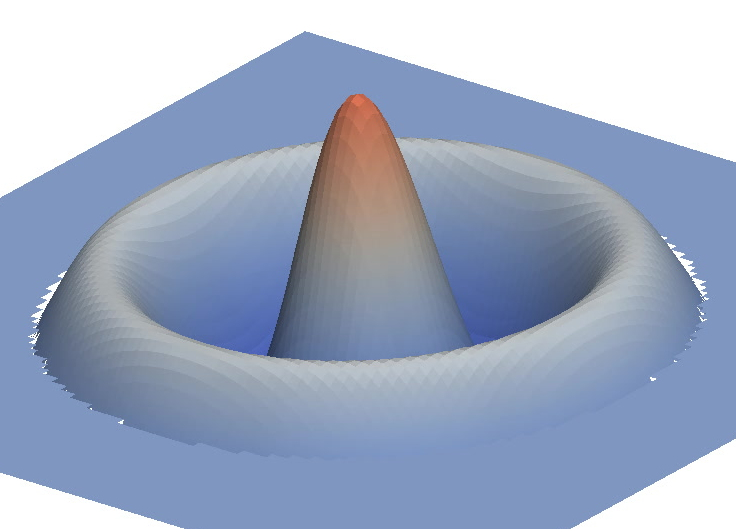}
	\centerline{$t\approx \frac{1}{8} T_p$}
  \end{minipage}
  \\
  \begin{minipage}{\innerSnapWidth}
	\includegraphics[width=\innerSnapWidth]{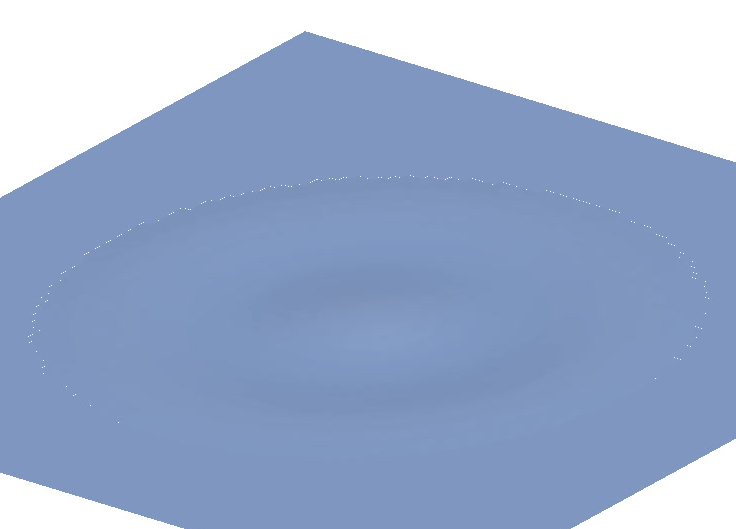}
	\centerline{$t\approx \frac{1}{4} T_p$}
  \end{minipage}
  \quad 
  \begin{minipage}{\innerSnapWidth}
	\includegraphics[width=\innerSnapWidth]{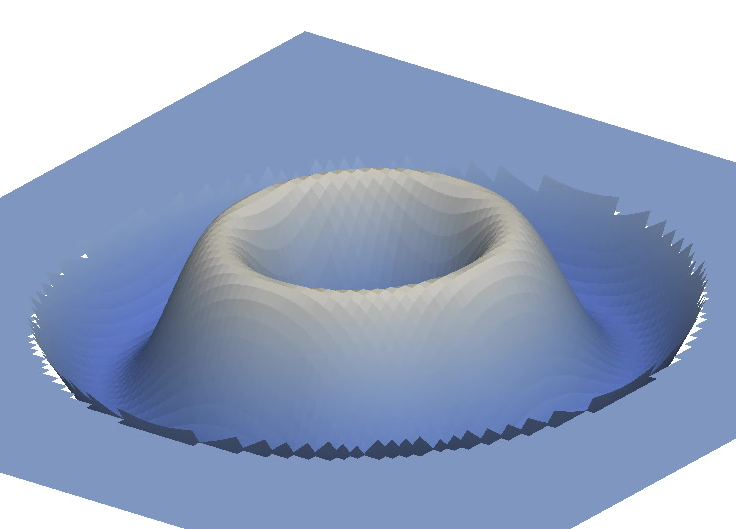}
	\centerline{$t\approx \frac{3}{8} T_p$}
  \end{minipage}
  \caption{Snapshots of the vibrating membrane problem\label{fig:innerSnapshots}}
\end{figure}

\makeatletter{}\subsubsection{Results Using the Space $V_h^p$}
The calculated errors and estimated convergence rates for the different element orders are shown in Tables \ref{tab:convergence_Q1} to \ref{tab:convergence_Q3}.
One would expect that the $L_2(\Omega)$- and $L_2(\partial \Omega)$-errors are proportional to $h^{p+1}$ and that the $H^1(\Omega)$-error is proportional to $h^p$.
The rates agree quite well with this.

Computed $C_{FL}$-constants for different element orders are shown in Table \ref{tab:cfl}.
The values were computed according to \eqref{eq:cflNumber}.
We see that the $C_{FL}$-constant is essentially the same as for the non-immersed case.
In the same way as for the non-immersed case, Table \ref{tab:cfl} displays the mean value over a number of grid sizes, but the $C_{FL}$-number only varied slightly when varying the grid size.
By inserting the values in Table \ref{tab:cfl} into \eqref{eq:time_step_bound} one can see that it would have been possible to use a larger time-step than the one in \eqref{eq:used_time_step}.

How the condition number of the mass matrix depend on the grid size is shown in Figure \ref{fig:condnumbers}, for the different orders of $p$.
We see that the condition numbers are essentially constant when refining $h$, in agreement with \eqref{eq:norm_equivalence}.
We also see that the condition numbers increase extremely rapidly when increasing the polynomial degree, as predicted by Lemma \ref{lem:conditioning}.
It is also clear from Figure \ref{fig:condnumbers} that the condition number increase much faster than in the non-immersed case.
The dashed lines in Figure \ref{fig:condnumbers} denote the function $C P(p)$, where $P$ is the function from \eqref{eq:Pfunction}.
The constant $C$ was chosen so that $CP(1)$ agreed with the mean (with respect to $h$) of the condition numbers for $V_h^1$.
The estimate from Lemma \ref{lem:conditioning} appears to be too pessimistic.

The minimal and maximal eigenvalues for the different polynomial orders and refinements are seen in Figure \ref{fig:EigenvaluesMassMatrix}.
As can be seen, the scaled eigenvalues are essentially constant with respect to $h$. 
Thus the dependence on $h$ is in agreement with the theoretical considerations in Section \ref{sec:estimates}.
We see that the minimal eigenvalues decrease quite fast when increasing the polynomial degree, and that they are substantially smaller than in the non-immersed case.
The maximal eigenvalues also decrease but much slower than in the non-immersed case.

\begin{table}[H]
\caption{Errors when using the space $V_h^1$\label{tab:convergence_Q1}}
\makeatletter{}\begin{center}\begin{tabular}{|c|c|c|c|c|c|c|}\hline$h$&\multicolumn{2}{|c|}{$L_{2}(\Omega)$}&\multicolumn{2}{|c|}{$H^{1}(\Omega)$)}&\multicolumn{2}{|c|}{$L_{2}(\partial\Omega)$}\\\hline1.200e-01&7.574e-02&-&1.354e+00&-&4.555e-02&-\\\hline6.000e-02&1.325e-02&2.52&5.494e-01&1.30&4.156e-03&3.45\\\hline3.000e-02&3.068e-03&2.11&2.692e-01&1.03&4.019e-04&3.37\\\hline1.500e-02&7.080e-04&2.12&1.340e-01&1.01&1.167e-04&1.78\\\hline\end{tabular}\end{center} 
\end{table}

\begin{table}[H]
\caption{Errors when using the space $V_h^2$\label{tab:convergence_Q2}}
\makeatletter{}\begin{center}\begin{tabular}{|c|c|c|c|c|c|c|}\hline$h$&\multicolumn{2}{|c|}{$L_{2}(\Omega)$}&\multicolumn{2}{|c|}{$H^{1}(\Omega)$)}&\multicolumn{2}{|c|}{$L_{2}(\partial\Omega)$}\\\hline1.200e-01&3.198e-03&-&1.561e-01&-&3.414e-03&-\\\hline6.000e-02&3.490e-04&3.20&3.640e-02&2.10&5.683e-04&2.59\\\hline3.000e-02&4.433e-05&2.98&8.897e-03&2.03&7.709e-05&2.88\\\hline1.500e-02&5.282e-06&3.07&2.141e-03&2.06&9.352e-06&3.04\\\hline\end{tabular}\end{center} 
\end{table}

\begin{table}[H]
\caption{Errors when using the space $V_h^3$\label{tab:convergence_Q3}}
\makeatletter{}\begin{center}\begin{tabular}{|c|c|c|c|c|c|c|}\hline$h$&\multicolumn{2}{|c|}{$L_{2}(\Omega)$}&\multicolumn{2}{|c|}{$H^{1}(\Omega)$)}&\multicolumn{2}{|c|}{$L_{2}(\partial\Omega)$}\\\hline1.200e-01&1.464e-04&-&1.181e-02&-&4.643e-05&-\\\hline6.000e-02&9.475e-06&3.95&1.412e-03&3.07&2.097e-06&4.47\\\hline3.000e-02&5.470e-07&4.11&1.707e-04&3.05&1.518e-07&3.79\\\hline1.500e-02&2.188e-08&4.64&2.304e-05&2.89&7.674e-09&4.31\\\hline\end{tabular}\end{center} 
\end{table}

\makeatletter{}\subsubsection{Results Using the Space $\tilde{V}_h^p$}
The errors and convergence rates when using the space $\tilde{V}_h^2$ are shown in Table \ref{tab:convergence_lower_boundary_Q2}.
For the errors in $L_2(\Omega)$- and $H^1(\Omega)$-norm it seems that we obtain the rate corresponding to the highest element ($Q_2$) in the space.
Not unexpectedly we seem to get the lower order convergence for the $L_2(\Omega)$-error along the boundary.
However, when using the space $\tilde{V}_h^3$ the situation appears to be different.
Here, it seems that one looses at least half an order for the rates of the $L_2(\Omega)$- and $H^1(\Omega)$-errors,
which is rather unsatisfactory.

How the condition number of the mass-matrix depends on $h$ for the two spaces $\tilde{V}_h^2$ and $\tilde{V}_h^3$ are shown in Figure \ref{fig:cond_lower_boundary}.
By comparing to Figure \ref{fig:condnumbers} we see that the condition number of the space $V_h^p$ is essentially the same as for the corresponding space with the lowest order element everywhere. That is:
\begin{equation*}
\kappa \left( \mathcal{M}_{\tilde{V}_h^p} \right) \approx \kappa \left( \mathcal{M}_{V_h^{p-1}} \right),
\end{equation*}
which is not surprising since we expect that all ill-conditioning is due to the added penalty term, $j(\cdot,\cdot)$.
The minimal and maximal eigenvalues of the mass matrix for the space $\tilde{V}_h^p$ look essentially the same as for the space $V_h^{p-1}$, 
while the CFL-numbers for the space $\tilde{V}_h^p$ look essentially the same as for the space $V_h^{p}$.

\begin{table}[H]
\caption{Errors when using the space $\tilde{V}_h^2$\label{tab:convergence_lower_boundary_Q2}}
\makeatletter{}\begin{center}\begin{tabular}{|c|c|c|c|c|c|c|}\hline$h$&\multicolumn{2}{|c|}{$L_{2}(\Omega)$}&\multicolumn{2}{|c|}{$H^{1}(\Omega)$)}&\multicolumn{2}{|c|}{$L_{2}(\partial\Omega)$}\\\hline1.200e-01&3.864e-02&-&7.468e-01&-&1.277e-02&-\\\hline6.000e-02&5.520e-03&2.81&1.589e-01&2.23&1.508e-03&3.08\\\hline3.000e-02&4.076e-04&3.76&2.361e-02&2.75&2.115e-04&2.83\\\hline1.500e-02&5.680e-05&2.84&6.046e-03&1.97&5.065e-05&2.06\\\hline\end{tabular}\end{center} 
\end{table}

\begin{table}[H]
\caption{Errors when using the space $\tilde{V}_h^3$\label{tab:convergence_lower_boundary_Q3}}
\makeatletter{}\begin{center}\begin{tabular}{|c|c|c|c|c|c|c|}\hline$h$&\multicolumn{2}{|c|}{$L_{2}(\Omega)$}&\multicolumn{2}{|c|}{$H^{1}(\Omega)$)}&\multicolumn{2}{|c|}{$L_{2}(\partial\Omega)$}\\\hline1.200e-01&1.765e-03&-&7.566e-02&-&1.457e-03&-\\\hline6.000e-02&2.610e-04&2.76&1.989e-02&1.93&2.444e-04&2.58\\\hline3.000e-02&2.350e-05&3.47&3.738e-03&2.41&3.773e-05&2.70\\\hline1.500e-02&2.382e-06&3.30&6.705e-04&2.48&3.880e-06&3.28\\\hline\end{tabular}\end{center} 
\end{table}

\begin{figure}[H]
\centering
\includegraphics[width=.35\paperwidth]{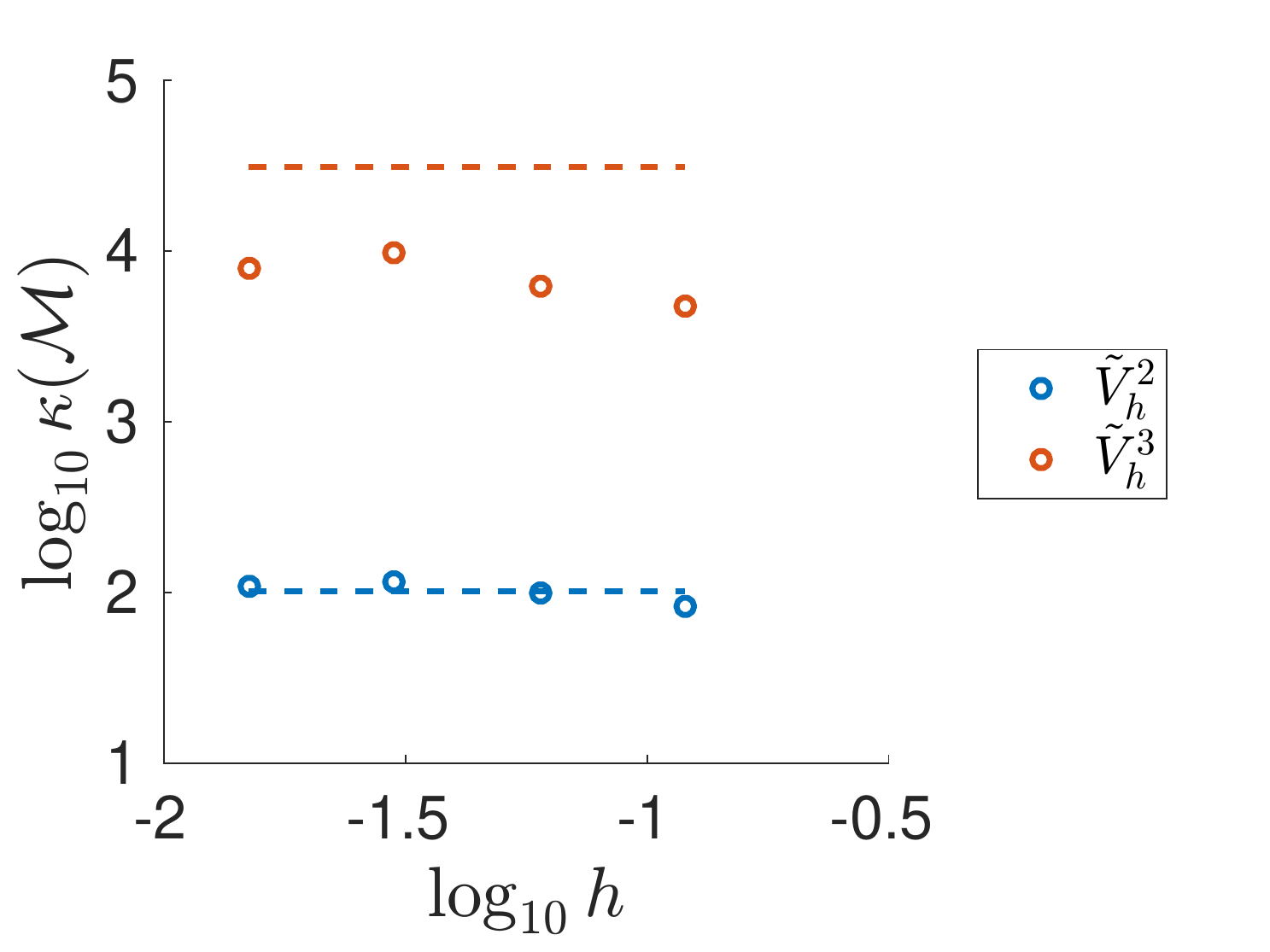}
\caption{Condition number of the mass matrix when using the space $\tilde{V}_h^p$, for different $h$ and $p$. The dashed lines
denotes estimates according to the function $P(p)$.}
\label{fig:cond_lower_boundary}
\end{figure} 
 
\makeatletter{}\subsection{An Immersed Outer Problem \label{sec:outerproblem}}
Consider instead an outer problem with the geometry depicted in Figure \ref{fig:geometryOuterProblem}. 
The star shaped geometry is the zero contour of the following level set function
\begin{equation*}
\phi(r,\theta)= R+R_0 \sin(n\theta)-r
\end{equation*}
where $(r,\theta)$ are the polar coordinates, and $R=0.5$, $R_0=0.1$, $n=5$.
So our domain $\Omega$ is given by
\begin{equation*}
\Omega=
\{(x,y)\in \R^2 : x,y\in (-3/2,3/2) : \phi(x,y)<0 \}. 
\end{equation*}

\begin{figure}[H]
\centering
  \includegraphics[width=.3\paperwidth]{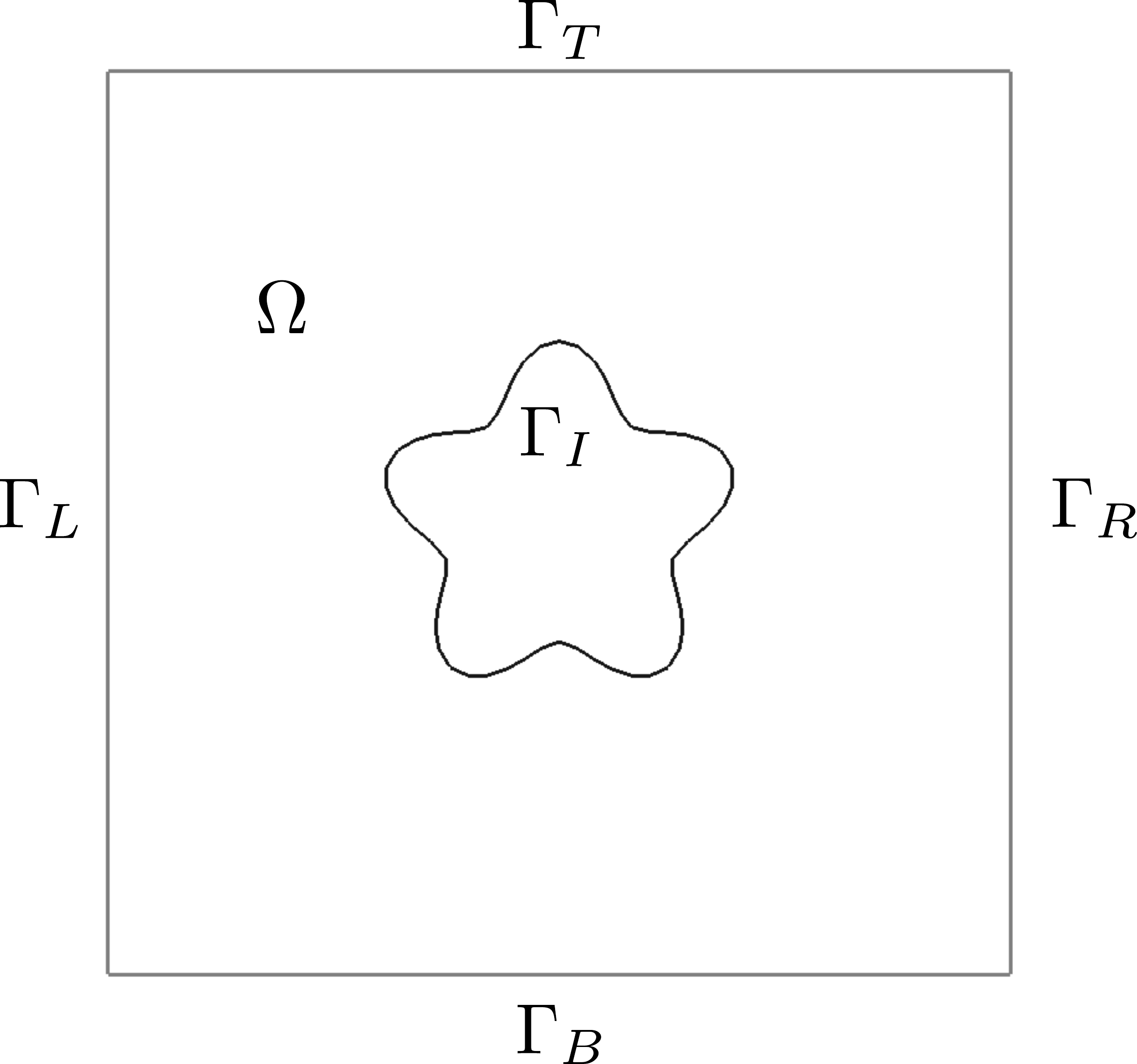}
  \caption{Geometry used for the outer problem}
  \label{fig:geometryOuterProblem}
\end{figure}
Starting from zero initial conditions
\begin{align*}
\evaluated{u}{t=0}=0, \\
\evaluated{\pdd{u}{t}}{t=0}=0,
\end{align*}
we prescribe homogeneous Neumann boundary condition on the internal boundary
\begin{equation*}
\evaluated{\pdd{u}{n}}{\Gamma_I}=0,
\end{equation*}
and Dirichlet boundary conditions on the external boundaries
\begin{equation*}
u=
\begin{cases}
g_D(x,t) &\quad x\in \Gamma_B \\
0 &\quad x\in \Gamma_L  \cup \Gamma_R \cup \Gamma_U.
\end{cases}		
\end{equation*}
Here, $g_D$ is the function
\begin{equation*}
g_D(x,t)= 
\cos \left( \frac{\pi}{3} x \right)
e^{-\left(\frac{t-t_c}{\sigma}\right)^2},
\end{equation*}
where we have chosen $\sigma=0.25$, $t_c=3$.
A few snapshots of the numerical solution is seen in Figure
\ref{fig:outer_snapshots}. 

Here, we don't have an expression for the analytic solution.
So when computing the errors we compare against a reference solution.
The reference solution was computed on a grid twice as fine as the finest grid that we present errors for.

Given the less satisfying results using the space $\tilde{V}_h^p$ in Section \ref{sec:innerproblem}, we here only examine the convergence results for the space $V_h^p$. 
The computed errors after solving to the end time $\finalTime=4$ are shown in Table \ref{tab:outer_Q1} to \ref{tab:outer_Q3}.

We see that the convergence is at least $h^{p+1}$ for the $L_2(\Omega)$-error and $h^{p}$ for the $H^1(\Omega)$-error.
In Table \ref{tab:outer_Q1} to \ref{tab:outer_Q3} the last column is the $L_2(\Gamma_N)$-error in the Neumann boundary condition
\begin{equation*}
\left \|\pdd{u_h}{n} -\pdd{u_h^{\text{ref}}}{n} \right \|_{\partial \Omega},
\end{equation*}
which we see is close to the expected rate $h^p$.

\begin{figure}[H]
\newcommand{\innerSnapWidth}{.25\paperwidth}
\centering
  \begin{minipage}{\innerSnapWidth}
	\includegraphics[width=\innerSnapWidth]{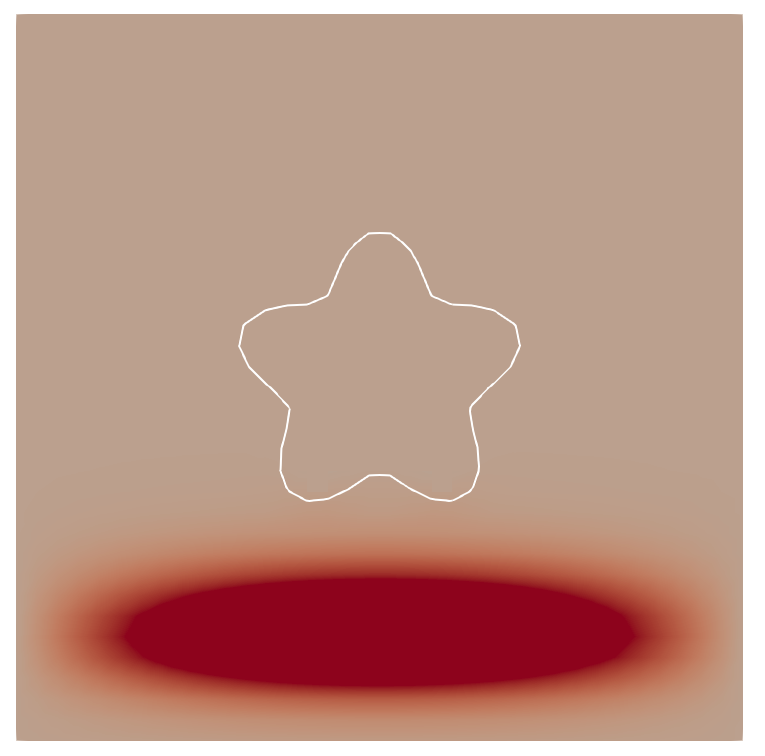}
	\centerline{$t\approx1.95$}
  \end{minipage}
  \quad 
  \begin{minipage}{\innerSnapWidth}
	\includegraphics[width=\innerSnapWidth]{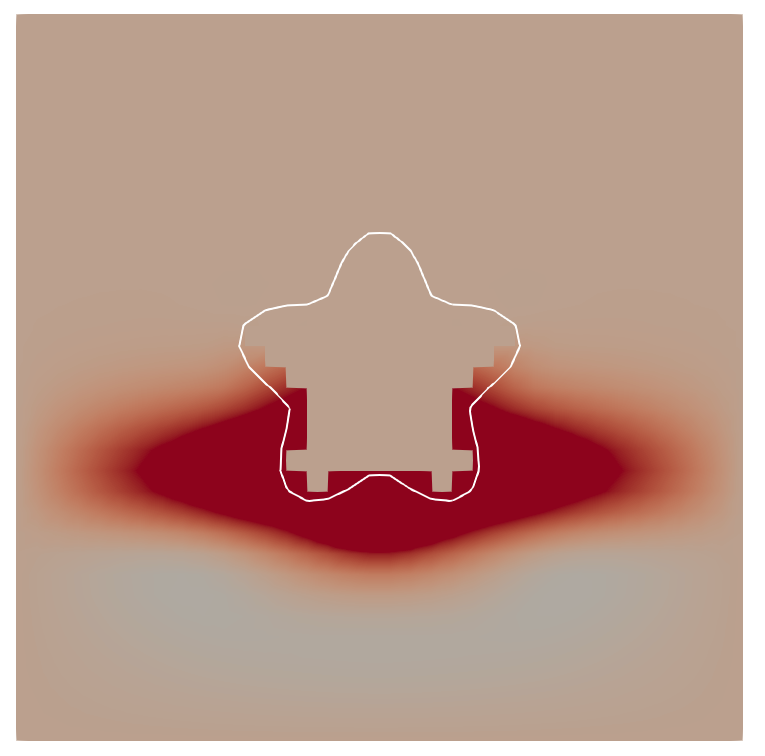}
	\centerline{$t\approx 2.63$}
  \end{minipage}
  \\
  \begin{minipage}{\innerSnapWidth}
	\includegraphics[width=\innerSnapWidth]{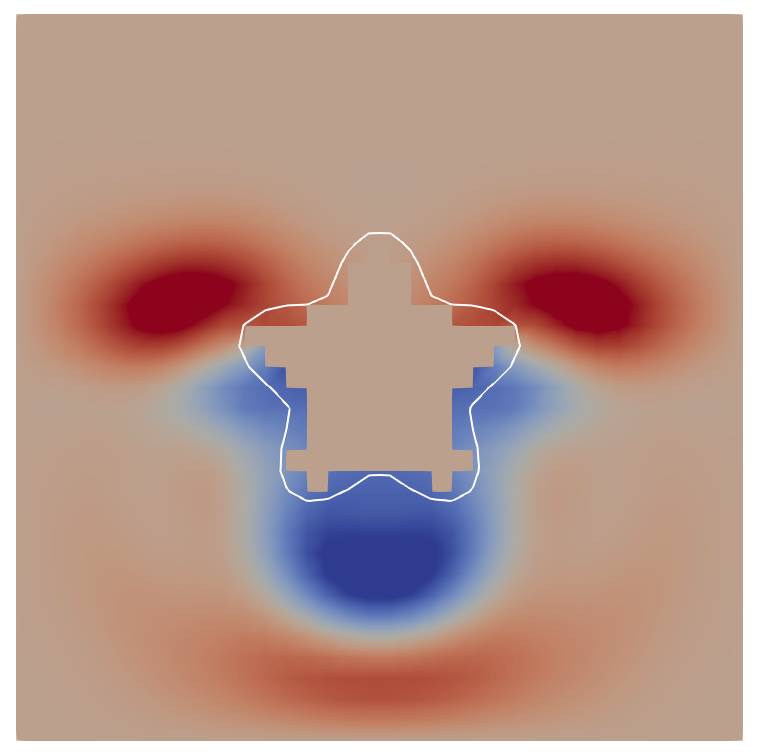}
	\centerline{$t\approx 3.32$}
  \end{minipage}
  \quad 
  \begin{minipage}{\innerSnapWidth}
	\includegraphics[width=\innerSnapWidth]{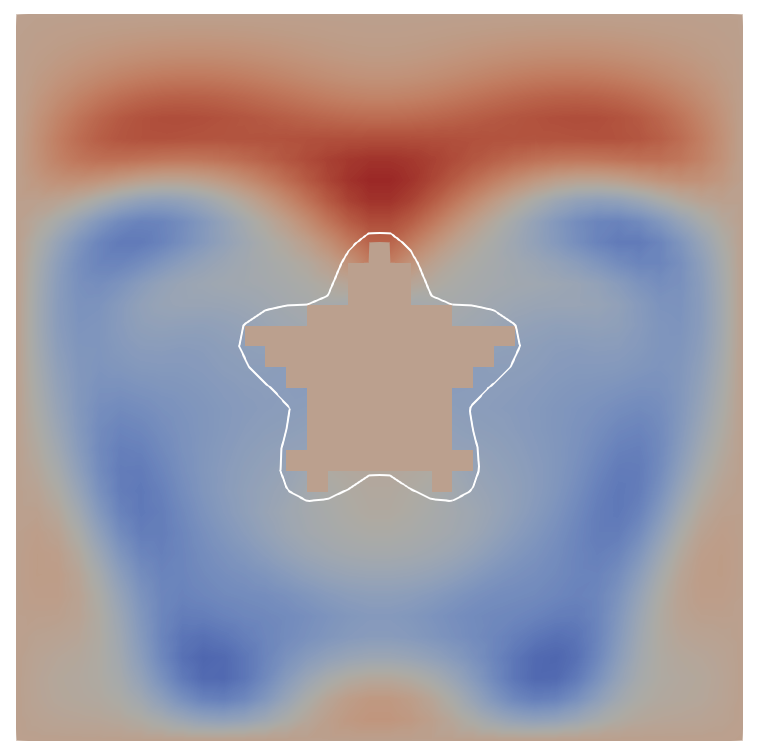}
	\centerline{$t=4$}
  \end{minipage}
  \caption{Snapshots of the numerical solution for the outer problem  \label{fig:outer_snapshots}}
\end{figure}

\begin{table}[H]
\caption{Errors for the outer problem when using the space $V_h^1$\label{tab:outer_Q1}}
\makeatletter{}\begin{center}\begin{tabular}{|c|c|c|c|c|c|c|}\hline$h$&\multicolumn{2}{|c|}{$L_{2}(\Omega)$}&\multicolumn{2}{|c|}{$H^{1}(\Omega)$)}&\multicolumn{2}{|c|}{$L_{2}(\partial\Omega)$}\\\hline1.500e-01&2.355e-01&-&2.048e+00&-&5.844e-01&-\\\hline7.500e-02&6.160e-02&1.93&6.724e-01&1.61&2.946e-01&0.99\\\hline3.750e-02&1.221e-02&2.34&1.952e-01&1.78&1.468e-01&1.00\\\hline\end{tabular}\end{center} 
\end{table}

\begin{table}[H]
\caption{Errors for the outer problem when using the space $V_h^2$\label{tab:outer_Q2}}
\makeatletter{}\begin{center}\begin{tabular}{|c|c|c|c|c|c|c|}\hline$h$&\multicolumn{2}{|c|}{$L_{2}(\Omega)$}&\multicolumn{2}{|c|}{$H^{1}(\Omega)$)}&\multicolumn{2}{|c|}{$L_{2}(\partial\Omega)$}\\\hline1.500e-01&3.335e-02&-&5.085e-01&-&5.956e-01&-\\\hline7.500e-02&1.805e-03&4.21&3.771e-02&3.75&1.925e-01&1.63\\\hline3.750e-02&1.060e-04&4.09&7.842e-03&2.27&4.159e-02&2.21\\\hline\end{tabular}\end{center} 
\end{table}

\begin{table}[H]
\caption{Errors when using the space $V_h^3$\label{tab:outer_Q3}}
\makeatletter{}\begin{center}\begin{tabular}{|c|c|c|c|c|c|c|}\hline$h$&\multicolumn{2}{|c|}{$L_{2}(\Omega)$}&\multicolumn{2}{|c|}{$H^{1}(\Omega)$)}&\multicolumn{2}{|c|}{$L_{2}(\partial\Omega)$}\\\hline1.500e-01&3.039e-03&-&9.497e-02&-&2.592e-01&-\\\hline7.500e-02&9.965e-05&4.93&3.837e-03&4.63&4.885e-02&2.41\\\hline3.750e-02&2.273e-06&5.45&4.548e-04&3.08&6.715e-03&2.86\\\hline\end{tabular}\end{center} 
\end{table} 
\makeatletter{}\section{Discussion \label{sec:discussion}}
The results in Section \ref{sec:innerproblem} and \ref{sec:outerproblem} show that it is possible to solve the wave equation and obtain up to 4th order convergence.
In particular, it is also promising that the CFL-condition is not stricter than for the non-immersed case.
However, both the theoretical results in Lemma \ref{lem:conditioning} and the results in Section \ref{sec:innerproblem} show that there are problems with the conditioning of the mass matrix. 
It should be emphasized that even if the added stabilization creates some new problems it is by far better than using no stabilization at all. 
With the added stabilization the method can be proved to be stable, which is essential.

It would, of course, be advantageous if one would be able to create a stabilization which does not lead to conditioning problems.
However, the prospects for creating a good preconditioner for the mass matrix is rather good, since the stabilization maintains the symmetry of the mass matrix and since one obtains bounds on its spectrum from the analysis.

The choice of the weights in \eqref{eq:mini_weights} were based on hand-waving 
arguments and can, therefore, be criticized. We have tried other choices of weights
but have not presented the results here. This is mainly because they give similar
results and we have no reason to believe that there exists a choice which makes the
condition number significantly better.

The idea of lowering the order of the elements close to the boundary worked quite well for the space $\tilde{V}_h^2$. 
We obtained the convergence corresponding to the higher elements in the space, but the condition number corresponding to the lower order elements.
Unfortunately, higher order convergence was not achieved when increasing the element order beyond 2.
Thus, the procedure does not appear to be a plausible solution for going to higher orders.

\bibliographystyle{spmpsci}
\bibliography{references}

\end{document}